\documentclass[11pt]{amsart}
\usepackage{amssymb,amsmath,amsthm}
\newcommand{\shrinkmargins}[1]{
  \addtolength{\textheight}{#1\topmargin}
  \addtolength{\textheight}{#1\topmargin}
  \addtolength{\textwidth}{#1\oddsidemargin}
  \addtolength{\textwidth}{#1\evensidemargin}
  \addtolength{\topmargin}{-#1\topmargin}
  \addtolength{\oddsidemargin}{-#1\oddsidemargin}
  \addtolength{\evensidemargin}{-#1\evensidemargin}
  }

\shrinkmargins{0.5}

\newcommand{\leg}[2]{\genfrac{(}{)}{}{}{#1}{#2}}

\newtheorem{theorem}{Theorem}
\newtheorem{lemma}[theorem]{Lemma}
\newtheorem{corollary}[theorem]{Corollary}

\theoremstyle{remark}

\newtheorem*{remark}{Remark}

\numberwithin{theorem}{section} \numberwithin{equation}{section}

\newcommand{\R}{\mathbb{R}}
\newcommand{\C}{\mathbb{C}}

\newcommand{\Z}{\mathbb{Z}}
\newcommand{\N}{\mathbb{N}}
\newcommand{\SL}{{\text {\rm SL}}}

\newcommand{\re}{\textnormal{Re}}

\def\H{\mathbb{H}}

\begin{document}
\title[Dyson's Rank, overpartitions, and weak Maass forms]{  Dyson's Rank, overpartitions, and weak Maass forms   }
\author{Kathrin Bringmann}
\address{School of Mathematics\\University of Minnesota\\ Minneapolis, MN 55455 \\U.S.A.}
\email{bringman@math.umn.edu}
\author{Jeremy Lovejoy}
\address{CNRS, LIAFA, Universit\'e Denis Diderot,
2, Place Jussieu, Case 7014, F-75251 Paris Cedex 05, FRANCE}
\email{lovejoy@liafa.jussieu.fr}

\date{\today}
\begin{abstract}
In a series of papers the first author and Ono connected  the
rank, a partition statistic introduced by Dyson, to weak Maass
forms, a new class of functions which are related to modular
forms. Naturally it is of wide interest to find other explicit
examples of Maass forms. Here we construct a new infinite family
of such forms, arising from overpartitions. As applications we
obtain combinatorial decompositions of Ramanujan-type congruences
for overpartitions as well as the modularity of rank differences
in certain arithmetic progressions.
\end{abstract}
\maketitle

\section{Introduction and Statement of Results}
A {\it partition} of a positive integer $n$ is any non-increasing
sequence of positive integers whose sum is  $n$. Let $p(n)$ denote
the number of partitions of $n$ (with the usual convention that
$p(0):=1$, and $p(n):=0$ for $n \not \in \N_0$).
%The partition
%function $p(n)$ has the well known infinite product generating
%function
%\begin{equation}\label{partgen}
%1+
%\sum_{n=1}^{\infty}p(n)q^n=\prod_{n=1}^{\infty} \frac{1}{1-q^n}.
%\end{equation}

Ramanujan proved that for every positive integer $n$, we have:
\begin{eqnarray} \label{ramanujan}
\begin{split}
p(5n+4)&\equiv 0\pmod 5,\\
p(7n+5)&\equiv 0\pmod 7,\\
p(11n+6)&\equiv 0\pmod{11}.\\
\end{split}
\end{eqnarray}
 In a celebrated paper Ono \cite{On1} treated these kinds of congruences
 systematically (also  see \cite{On3}). Combining Shimura's theory of modular
 forms of half-integral weight with results of Serre on modular forms modulo
 $\ell$ he showed that for any prime $\ell \geq 5$ there exist infinitely
 many non-nested arithmetic progressions  of the form $An+B$  such that
\begin{eqnarray*}
p(An+B) \equiv 0 \pmod \ell.
\end{eqnarray*}

In order to explain the congruences  in (\ref{ramanujan}) with
moduli $5$ and $7$ combinatorially, Dyson \cite{Dy} introduced the
rank of a partition. The \textit{rank} of a partition is defined
to be its largest part minus the number of its parts. Dyson
conjectured that the partitions of $5n+4$ (resp. $7n+5$) form $5$
(resp. $7$) groups of equal size when sorted by  their ranks
modulo $5$ (resp. $7$). This conjecture was proven in 1954 by
Atkin and Swinnerton-Dyer \cite{AS}. In \cite{BO2} and \cite{B2},
Ono and the first author showed that Dyson's rank partition
function also satisfies congruences of Ramanujan type.  One of the
main steps in their proof is  to show that generating functions
related to the rank are the ``holomorphic parts'' of ``weak Maass
forms", a notion we will explain later.  This new theory has many
applications, such as congruences \cite{BO2,B2} and asymptotics
\cite{B1} for ranks as well as modularity for rank differences
\cite{BOR}.

Naturally it is of wide interest to find other explicit examples
of weak Maass forms.  After partitions, the next place to look is
overpartitions.  Recall that an \textit{overpartition} is a
partition where the first occurrence of a summand may be overlined
(see \cite{CL}). For example, there are $14$ overpartitions of
$4$:
$$
\begin{gathered}
4, \overline{4}, 3+1, \overline{3} + 1, 3 + \overline{1},
\overline{3} + \overline{1}, 2+2, \overline{2} + 2,  \\
2+1+1, \overline{2} + 1 + 1, 2+ \overline{1} + 1, \overline{2} +
\overline{1} + 1, 1+1+1+1, \overline{1} + 1 + 1 +1.
\end{gathered}
$$
Overpartitions have arisen in many areas where ordinary partitions
play an important role, most notably in $q$-series and
combinatorics (e.g. \cite{Be-Pa1,Co-Hi1,CL.5,CL,Lo2,Pa1,Ye1}), but
also in mathematical physics (e.g. \cite{Fo-Ja-Ma1,Fo-Ja-Ma2}),
symmetric functions (e.g. \cite{Bre1,De-La-Ma3}, representation
theory (e.g. \cite{Ka-Kw1}) and algebraic number theory (e.g.
\cite{Lo1,Lo-Ma2}).  To give a few specific examples, the
combinatorial theory of overpartitions leads to natural and
straightforward bijective proofs of $q$-series identities like
Ramanujan's $_1\psi_1$ summation \cite{CL.5,Ye1}; in the theory of
symmetric functions in superspace, overpartitions play the role
that partitions play in the classical theory of symmetric
functions \cite{Fo-Ja-Ma1,Fo-Ja-Ma2}; and certain Dedekind zeta
functions associated to rings of integers of real quadratic fields
can be regarded as generating functions for weighted counts of
overpartitions \cite{Lo1,Lo-Ma2}.

Returning to Dyson's rank, this statistic applies just as well to
overpartitions.  Indeed, this rank and its generalizations have
already proven fundamental in the combinatorial theory of
overpartitions \cite{Co-Ma1,Lo,Lo-Ma1}.  The main result of the
present paper will be the construction of an infinite family of
weak Maass forms whose holomorphic parts are related to the
generating function for Dyson's rank of an overpartition.  As
applications, we discuss congruence properties of overpartitions
and the modularity of rank differences in arithmetic progressions.

For a positive integer $n$ we denote by $\overline{p}(n)$ the
number of overpartions of $n$. We have the  generating function
\cite{CL}
\begin{eqnarray} \label{overgenerating}
\overline{P}(q):= \sum_{n \geq 0}  \overline{p}(n) \, q^n
= \frac{\eta(2z)}{\eta^2(z)}=
1+ 2 q + 4q^2 +8q^3+ 14 q^4 +  \cdots .
\end{eqnarray}
Here $\eta(z):=q^{\frac{1}{24}} \prod_{n=1}^{\infty} \left( 1-q^n
\right)$ is Dedekind's eta function and we write $q:=e^{2 \pi  i
z}$. Moreover we denote by $\overline{N}(m,n)$ the number of
overpartitions of $n$ with rank $m$. It is shown in \cite{Lo} that
\begin{equation}  \label{overrankgen}
\begin{split}
\mathcal{O}(u;q)&:= 1 + \sum_{n = 1}^{\infty} \overline{N}(m,n) u^m q^n
= \sum_{n = 0}^{\infty}
\frac{(-1)_{n}   q^{\frac{1}{2}n (n+1)}}{(uq,q/u)_n}\\
&=
\frac{(-q)_{\infty}}{(q)_{\infty}}
\left(
1 + 2 \sum_{n \geq 1}
\frac{\left(  1-u\right)  \left(1-u^{-1}   \right)(-1)^n q^{n^2+n}}{\left(1- u q^n   \right)  \left(  1- u^{-1} q^n \right)}
\right).
\end{split}
\end{equation}
Here for $a,b \in \C$, $n \in \N \cup \{\infty\}$, we employ the
standard $q$-series notation:
\begin{eqnarray*}
(a)_{n}:&=& \prod_{r=0}^{n-1} \left(1-aq^r\right),\\
(a,b)_{n}:&=& \prod_{r=0}^{n-1} \left(1-aq^r\right)  \left(1-bq^r\right),\\
(a)_{\infty}:&=&\lim_{n \to \infty} (a)_n.
\end{eqnarray*}
It turns out that the function $\mathcal{O}(u;q)$ for $u$ a root
of unity $\not=1$ is the holomorphic part of a weak Maass form.

To make this precise, we recall the notion of a  weak Maass form
of half-integral weight $k\in \frac{1}{2}\Z\setminus \Z$. If
$z=x+iy$ with $x, y\in \R$, then the weight $k$ hyperbolic
Laplacian is given by
\begin{equation}\label{laplacian}
\Delta_k := -y^2\left( \frac{\partial^2}{\partial x^2} +
\frac{\partial^2}{\partial y^2}\right) + iky\left(
\frac{\partial}{\partial x}+i \frac{\partial}{\partial y}\right).
\end{equation}
If $v$ is odd, then define $\epsilon_v$ by
\begin{equation}
\epsilon_v:=\begin{cases} 1 \ \ \ \ &{\text {\rm if}}\ v\equiv
1\pmod 4,\\
i \ \ \ \ &{\text {\rm if}}\ v\equiv 3\pmod 4. \end{cases}
\end{equation}
 A {\it (harmonic) weak Maass form}   of weight $k$ and Nebentypus $\chi$ on a subgroup
$\Gamma \subset \Gamma_0(4)$ is any smooth function $f:\H\to \C$
satisfying the following:
\begin{enumerate}
\item For all $A= \left(\begin{smallmatrix}a&b\\c&d
\end{smallmatrix} \right)\in \Gamma$ and all $z\in \H$, we
have
\begin{displaymath}
f (Az)= \leg{c}{d}^{2k}\epsilon_d^{-2k}\chi(d)(cz+d)^{k}\ f(z).
\end{displaymath}
\item We  have that $\Delta_k f=0$.
\item The function $f(z)$ has
at most linear exponential growth at all the cusps of $\Gamma$.
\end{enumerate}
\smallskip

Suppose that $0<a<c$ are  integers,  and let $\zeta_c:=e^{\frac{2\pi
i}{c}}$.
Define the theta function of weight $\frac32$
\begin{equation}\label{shimuratheta}
\theta(\alpha,\beta;\tau):=
\sum_{n\equiv \alpha \pmod \beta}
n e^{\frac{2\pi i \tau n^2}{2 \beta}},
\end{equation}
and let
 \begin{eqnarray*}
\Theta_{a,c}(\tau):=
\left\{
\begin{array}{ll}
\theta \left( 4a + c, 2c; \frac{\tau}{4c}\right)&\text{if } c \text{ is odd},\\
2 \theta \left( 2a + \frac{c}{2}, c; \frac{\tau}{2c}\right)&\text{if } 2\parallel c,\\
4 \theta \left( a + \frac{c}{4}, \frac{c}{2}; \frac{\tau}{c}\right)&\text{if } 4| c.
\end{array}
\right.
\end{eqnarray*}
Using these cuspidal theta functions,  we define for $c\not=2$, the non-holomorphic  integral
\begin{equation}
J\left (\frac{a}{c};z\right)
 := \frac{ \pi  i \cdot \tan\left (\frac{\pi
a}{c}\right)}{  4 c   }
\int_{- \bar z}^{i
\infty} \frac{(- i \tau)^{-\frac{3}{2}}   \cdot \Theta_{a,b}\left (   -\frac{1}{ \tau}   \right) }
{\sqrt{-i(\tau+z)}} \ d\tau.
\end{equation}
Moreover   define $\mathcal{M}\left(\frac{a}{c};z \right)$ by
\begin{equation}
\mathcal{M}\left (\frac{a}{c};z\right )
:=
\mathcal{O} \left(\frac{a}{c}  ;q  \right)
-   J \left (\frac{a}{c};z\right ) ,
\end{equation}
where
$\mathcal{O}\left( \frac{a}{c};q \right):=\mathcal{O} \left( \zeta_c^a;q \right)$.
If  $u=-1$, we   define
\begin{eqnarray*}
  \mathcal{M}(-1;z):= \mathcal{O}(-1;z) - I(-1;z).
  \end{eqnarray*}
  with
  \begin{eqnarray*}
  I(-1;z):= \frac{\sqrt{2}}{\pi i} \int_{-\bar z}^{i \infty}
   \frac{\eta^2(\tau)}{\eta(2\tau) \cdot (-i (\tau+z))^{\frac32}}
  \, d \tau.
  \end{eqnarray*}
  The main result of this paper is the following theorem which establishes that those
  real analytic functions are Maass forms.
\begin{theorem}\label{maassform}
The following statements are true:
\begin{enumerate}
\item If $0<a<c$ with  $(a,c)=1$ and $c \not =2$, then $
\mathcal{M}\left (\frac{a}{c};z\right )   $ is a weak Maass form
of weight $\frac{1}{2}$ on $\Gamma_1(16c^2)$. If $2|c$ and $4|c$,
then it is a weak Maass form on   $\Gamma_1(4c^2)$ and
$\Gamma_1(c^2)$, respectively. \item The function
$\mathcal{M}(-1;z)$ is a weak Maass form of weight $\frac32$ on
$\Gamma_0(16)$.
\end{enumerate}
\end{theorem}
\noindent {\it Five remarks.}

\noindent 1) If $c$ is odd, we actually obtain Maass forms for the
larger group
$$
\left\{  \left.
\left(
\begin{smallmatrix}
\alpha & \beta\\
\gamma& \delta
\end{smallmatrix}
\right) \in \SL_2(\Z) \right| \alpha \equiv \delta \equiv 1 \pmod{4c},\, \gamma \equiv 0 \pmod{16c^2}
\right\}.
$$

\noindent 2)   The proof of the second part of  Theorem
\ref{maassform} is harder than the first since the generating
function has double poles.  To overcome this problem, we introduce
new functions $\mathcal{O}_r(q)$ having an additional parameter
$r$ but only simple poles such that one can obtain
$\mathcal{O}(-1;z)$ by a process of differentiation.  This
differentiation accounts for the augmentation of the weight by $1$
in this case.  It is worth mentioning that for the case of the
classical Dyson's rank generating functions in \cite{BO2}, the
weak Maass forms have weight $1/2$ for every root of unity $\neq
1$.

\noindent 3)
  The authors \cite{BL} show that in the context of overpartition
  pairs, the analogous generating functions associated to the
  appropriate generalization of Dyson's rank are not weak Maass forms, but
  classical modular forms.

\noindent 4) We should stress that the analysis of the
transformation behavior of $\mathcal{O}(u;q)$ is much more
involved than in the case of the Dyson's rank generating functions
in \cite{BO2}. One of the reasons is that the half-integer weight
modular form $\frac{(-q)_{\infty}}{(q)_{\infty}}$ that shows up in
(\ref{overrankgen}) is not mapped to itself as in the case of the
usual ranks.  This prohibits ``guessing'' images under M\"obius
transformations as in \cite{BO2}. There the first author and Ono
started with part of images of the generating function that they
were able to guess.  Thus the idea of proof in \cite{BO2} which
builds on old results of Watson, cannot be employed here. Instead
we have to determine explicitly the images under all M\"obius
transformations with different techniques.

In view of Theorem \ref{maassform} one can obtain results on
overpartitions by arguing as in work of Ono and the first author
\cite{B1,B2,BO1,BO2,BOR}. In this direction we exhibit congruences
for $\overline{N}(r,t;n)$, the number of overpartitions of $n$
whose rank is congruent to $r \pmod t$, and provide a theoretical
framework for proving identities for rank differences in
arithmetic progressions.  Other possible applications, which we do
not address here, would be to asymptotics or inequalities for
ranks, exact formulas or distribution questions. We first consider
congruences satisfied by $\overline{N}(r,t;n)$. For ease of
notation we restrict to the case that $t$ is odd, the case $t$
even  can be considered similarly.
\begin{theorem}\label{congruences}
Let $t$ be a positive odd integer, and let $\ell \nmid 6t$ be a  prime.
If $j$ is a positive integer, then there are infinitely many
non-nested arithmetic progressions $An+B$ such that for every
$0\leq r <t$ we have
$$
\overline{N}(r,t;An+B)\equiv 0\pmod{\ell^j}.
$$
\end{theorem}
\begin{theorem}  \label{congruences2}
Suppose that $\ell \geq 5$ is a prime, $m,u,\beta \in \N$ with $\leg{-\beta}{\ell}=-1$.
Then a positive proportion of primes $p \equiv -1 \pmod{\ell}$ have the property that for every
$0 \leq r \leq \ell^m -1$
$$
\overline{N}\left(r,\ell^m;p^3 n \right) \equiv 0 \pmod{\ell^u}
$$
for all $ n \equiv \beta \pmod{ \ell}$ that  are not divisible by $p$.
\end{theorem}
This directly implies.
\begin{corollary} \label{cor1}
If $ \ell \geq 5$ is a prime, $m,u \in \N$, then there are
infinitely many non-nested arithmetic progressions $An+B$ such
that
$$
\overline{N} \left(r,\ell^m;An+B \right) \equiv 0 \pmod{\ell^u}
$$
for all $ 0 \leq r \leq \ell^m -1$.
\end{corollary}
\noindent
{\it Remark.}

\noindent
The congruences in Theorems \ref{congruences} and \ref{congruences2}
may be viewed as a combinatorial
decomposition of the overpartition function congruence
\begin{eqnarray} \label{overpartcong}
\overline{p}(An+B)\equiv 0\pmod{\ell^u}.
\end{eqnarray}
That (\ref{overpartcong}) holds for  infinitely many non-nested
arithmetic progressions $An+B$ was first observed by  Treneer
\cite{Tr}.

We next put identities involving rank differences for
overpartitions in the framework of weak Maass forms (see also
\cite{BOR}). For this define for a prime $\ell$ and integers $s_1$
and $s_2$ the function
 \begin{eqnarray*}
R_{s_1,s_2}(d) :=
\sum_{n =0}^{\infty} \left(\, \overline{N}(s_1,\ell,\ell n+d) - \overline{N}(s_2,\ell,\ell n+d) \right)\, q^{\ell n+d}.
\end{eqnarray*}
We provide a framework that could be used to show  an infinite
family of identities (see also \cite{BOR} for related results for
usual ranks).
\begin{theorem} \label{identities}
If  $\leg{d}{\ell} = - \leg{-1}{\ell}$,
then the function $R_{s_1,s_2}(d)$ is a weakly holomorphic modular  form on
$\Gamma_1 \left( 16 \ell^4\right)$.
\end{theorem}
Using Theorem  \ref{identities}, we could  prove concrete identities using   the valence formula.
 Since  the computations are straightforward but  lengthy  (coming from the fact that $\Gamma_1\left( 16 \ell^4\right)$  has a lot of cusps),
 we chose not to  prove individual identities.
Instead we just list some identities, and  their truth follows from work of the  second  author and Osburn \cite{LO}.
%Next we turn to identities satisfied by generating functions for rank differences. For this define for a positive integer $N$, $g,h$ real numbers that are not simultaneously congruent to $0 \pmod N$,
%the generalized Dedekind eta-function
%\begin{eqnarray*}
%E_{g,h}(z) :=
%q^{ \frac{1}{2} B \left(\frac{g}{N}  \right)}
%\cdot
%\prod_{m=1}^{\infty}
%\left(
%1-  \zeta_{N}^h \cdot  q^{m-1+ \frac{g}{N}}
%\right)
%\left(
%1-  \zeta_{N}^{-h} \cdot  q^{m- \frac{g}{N}}
%\right),
%\end{eqnarray*}
%where $B(x):=x^2-x+ \frac{1}{6}$.  We recall the modularity properties of this function in Section  \ref{SectionIdentities}.

The paper is organized as follows. In Section  \ref{SectionTrans},
we prove a transformation law for the rank generating functions in
the case $c \not=2$. In Section \ref{SectionMaass}, we show the
first part of Theorem \ref{maassform}. The main step is to
recognize  the Mordell type integrals occurring the transformation
law of the rank generating functions as integrals of theta
functions. In Section \ref{section-1}  we treat the case $c=2$
which is more  complicated  due to double poles of the generating
function. In Sections \ref{SectionCong1} and \ref{SectionCong2} we
show congruences for $\overline{N}(r,t,n)$. Section
\ref{SectionIdentities} is dedicated the proof of Theorem
\ref{identities}.
%%%%%%%%%%%%%%%%%%%%%%%%%%%%
%%%%%%%%%%%%%%%%%%%%%%%%
%%%%%%%%%%%%%%%%%%
\section*{Acknowledgements}
The authors thank the referee for many helpful suggestions which improved the exposition of the paper.   %%%%%%%%%%%%%
%%%%%%%
\section{A transformation law}  \label{SectionTrans}
Here we consider modularity properties for $\mathcal{O}\left(\frac{a}{c};q  \right)$.
For this  we need some  notation.
Let $c>2$ and $k$ be positive integers. Let $\widetilde k$ be either   $0$ or $1$ depending on whether  $k$ is even or   odd. Moreover let
$k_1:=\frac{k}{(k,c)}$, $c_1= \frac{c}{(c,k)}$, and define the integer $0\leq l <c_1$ by the congruence
$ l \equiv a k_1 \pmod{c_1}$.
If $\frac{b}{c} \in (0,1)$, then
define the integers $s(b,c)$ and $t(b,c)$ (for $\frac{b}{c}\not=\frac{1}{2}$)
by
\begin{eqnarray*}
s(b,c):=\begin{cases}
0 \ \ \ \ \ &{\text {\rm if}}\ 0<\frac{b}{c}\leq \frac{1}{4},\\
1 \ \ \ \ \ &{\text {\rm if}}\ \frac{1}{4} <  \frac{b}{c}\leq \frac{3}{4},\\
2 \ \ \ \ \ &{\text {\rm if}}\ \frac{3}{4}< \frac{b}{c}<1,
\end{cases}     \qquad
t(b,c):=\begin{cases}
1 \ \ \ \ \ &{\text {\rm if}}\ 0<\frac{b}{c}<\frac{1}{2},\\
3 \ \ \ \ \ &{\text {\rm if}}\ \frac{1}{2}< \frac{b}{c}<1.
\end{cases}
\end{eqnarray*}
In particular let $s:= s(l,c_1)$ and $t:= t(l,c_1)$.
Let $h'$ be defined by $h h' \equiv -1 \pmod k$.
Moreover let $\omega_{h,k}$ be given by
\begin{eqnarray} \label{omega}
\omega_{h,k} :=
\exp\left(\pi i
 \sum_{\mu \pmod k}  \left( \left( \frac{\mu}{k}\right) \right)   \left( \left( \frac{h \mu}{k}\right) \right)
  \right),
\end{eqnarray}
where
\begin{eqnarray*}
((x)):= \left \{
\begin{array}{ll}
x- \lfloor x \rfloor - \frac{1}{2} &\text{if } x \in \R \setminus \Z ,\\
0&\text{if } x \in \Z.
\end{array}
\right.
\end{eqnarray*}
Define    for $q=e^{2 \pi i z}$  the following functions.
\begin{eqnarray*}
\mathcal{U}\left (\frac{a}{c};q\right)
&:=&  \mathcal{U}\left (\frac{a}{c};z\right) :=
\sin \left(\frac{\pi a}{c} \right)
\frac{\eta\left(\frac{z}{2}\right)}{\eta^2(z)}  \sum_{n  \in \Z }
\frac{ \left( 1+q^n \right)  q^{n^2+\frac{n}{2}}}
{1 - 2 q^n \cos  \left(\frac{2 \pi a}{c} \right)+q^{2n}} ,\\
\mathcal{U}(a,b,c;q)
&:= &  \mathcal{U}(a,b,c;z):=
\frac{\eta\left(\frac{z}{2} \right)}{\eta^2(z)}
e^{ \frac{\pi i a}{c}\left(\frac{4b}{c} -1 -2 s(b,c)\right)}
q^{ \frac{s(b,c)}{c} + \frac{b}{2c}-\frac{b^2}{c^2}}
\sum_{m \in \Z}
\frac{q^{\frac{m}{2}(2m+1) + m s(b,c) }}{1-e^{-\frac{2\pi i a}{c}}\, q^{m + \frac{b}{c} } }, \\
%e^{\frac{2 \pi ia}{c} \left(\frac{2b}{c} - s(b,c)  \right) } \cdot
%q^{\frac{s(b,c) b}{c} - \frac{b^2}{c^2}} \\
%& &
%\left(
%\sum _{m=0}^{\infty}
%\frac{e^{-\frac{\pi i a}{c}}\cdot q^{\frac{m}{2}(2m+1) + ms(b,c) + \frac{b}{2c}}}{1- e^{-\frac{2 \pi i a}{c}} \cdot q^{m+\frac{b}{c}}}
%-
%\sum _{m=1}^{\infty}
%\frac{ e^{\frac{\pi i a}{c}}\cdot q^{\frac{m}{2}(2m+1) - ms(b,c) - \frac{b}{2c}}}{1- e^{\frac{2 \pi i a}{c}} \cdot q^{m -  \frac{b}{c}}}
%\right),
%\end{eqnarray*}
%\begin{eqnarray*}
\mathcal{V}(a,b,c;q)
&:=&   \mathcal{V}(a,b,c;z):=
\frac{\eta\left(\frac{z}{2}\right)}{\eta^2(z)}
e^{\frac{\pi ia}{c} \left(\frac{4b}{c} -1-2 s(b,c)  \right) } \cdot
q^{\frac{s(b,c) b}{c} +\frac{b}{2c} - \frac{b^2}{c^2}}
\sum _{m\in \Z}
\frac{q^{\frac{m}{2}(2m+1) + ms(b,c)}}{1- e^{-\frac{2 \pi i a}{c}} \cdot q^{m+\frac{b}{c}}},\\
\mathcal{O}(a,b,c;q)
&:=& \mathcal{O}(a,b,c;z):=
\frac{\eta(2 z)}{\eta^2(z)}
e^{\frac{\pi ia}{c} \left(\frac{4b}{c} -1- t(b,c)  \right) } \cdot
q^{\frac{t(b,c)b}{2c} +\frac{b}{2c}- \frac{b^2}{c^2}}
\sum _{m \in \Z}
(-1)^m
\frac{q^{\frac{m}{2}(2m+1) + \frac{mt(b,c)}{2}}  }{1- e^{-\frac{2 \pi i a}{c}} \cdot q^{m+\frac{b}{c}}}
,\\
%\end{eqnarray*}
% \begin{eqnarray*}
\mathcal{V}\left(\frac{a}{c};q\right)
&:= &
  \mathcal{V}\left(\frac{a}{c}; z \right) :=
\frac{\eta(2 z)}{\eta^2(z)}
q^{\frac{1}{4} }
\sum _{m \in \Z}
\frac{ q^{m^2 + m } \cdot \left(1+ e^{-\frac{2 \pi i a}{c} } \cdot q^{m + \frac{1}{2}} \right)}{1- e^{-\frac{2 \pi i a}{c}} \cdot q^{m +  \frac{1}{2}}} .
%-
%\sum _{m=1}^{\infty}
%\frac{q^{m^2 -m }   \cdot \left(1+ e^{\frac{2 \pi i a}{c} } \cdot q^{m - \frac{1}{2}} \right)}{1- e^{\frac{2 \pi i a}{c}} \cdot q^{m-\frac{1}{2}}}
%\right).
\end{eqnarray*}
Moreover let
\begin{eqnarray*}
H_{a,c}(x):= \frac{e^x }{1- 2 \cos \left(\frac{2 \pi a}{c}  \right)e^x + e^{2x}}.
\end{eqnarray*}
Then
\begin{eqnarray*}
H_{a,c}(-x)&=& H_{a,c}(x),\\
H_{a,c}(x + 2 \pi i )&=& H_{a,c}(x).
\end{eqnarray*}
Moreover define for an integer $\nu$ the Mordell type integral
\begin{eqnarray*}
I_{a,c,k,\nu} (w):=
\int_{\R} e^{- \frac{2 \pi  w x^2}{k}}
 H_{a,c} \left( \frac{2 \pi i \nu}{k}  - \frac{2 \pi w x}{k}  -  \frac{\widetilde k \pi i }{2k}\right)\, dx.
\end{eqnarray*}
For $k$ even we have to take the principal part of the integral.
We are now ready to show  the transformation law of $\mathcal{O}\left(\frac{a}{c};q  \right)$.
\begin{theorem} \label{Otransfo}
Assume the notation above. Moreover, let
$w \in \C$ with $\re(w)>0$,
$q:=e^{\frac{2 \pi i }{k} (h+iw)}$, and    $q_1:=e^{\frac{2 \pi i }{k} \left(h'+\frac{i}{w}\right)}$.
\begin{enumerate}
\item   If $c|k$ and $k$ is even, then we have
\begin{multline*}
\mathcal{O} \left( \frac{a}{c};q\right)=
(-1)^{k_1}   i   \cdot e^{- \frac{2 \pi i a^2 h'k_1}{c}} \cdot
 \tan \left( \frac{\pi a}{c} \right) \cdot   \cot \left( \frac{\pi a h'}{c}\right)   \frac{\omega_{h,k}^2}{ \omega_{h,k/2} } \,
 w^{-\frac{1}{2}}  \cdot
 \mathcal{O} \left(\frac{ah'}{c};q_1  \right)
\\  +
  \frac{4 \cdot \sin^2\left( \frac{\pi a }{c}  \right) \cdot   \omega_{h,k}^2}{\omega_{h,k/2} \cdot k }
\cdot  w^{\frac{1}{2}}
 \sum_{\nu \pmod k} (-1)^{\nu}\,
 e^{ - \frac{2\pi i h' \nu^2}{k} }
    \cdot
 I_{a,c,k,\nu}(w).
\end{multline*}
\item  If $c|k$ and $k$ is  odd, then we have
\begin{multline*}
\mathcal{O} \left( \frac{a}{c};q\right)= \sqrt{2}
 i \cdot e^{\frac{\pi i h'}{8k}  - \frac{2 \pi i a^2 h'k_1}{c}}
 \tan \left( \frac{\pi a}{c} \right)  \frac{\omega_{h,k}^2}{\omega_{2h,k} } \cdot
 w^{-\frac{1}{2}}  \cdot
 \mathcal{U} \left(\frac{ah'}{c};q_1  \right)  \\
 +    \frac{4\sqrt{2}  \cdot \sin^2\left( \frac{\pi a }{c}  \right) \cdot   \omega_{h,k}^2}{\omega_{2h,k}\cdot k }
\cdot  w^{\frac{1}{2}}
 \sum_{ \nu  \pmod k} e^{-\frac{\pi i  h'}{k} (2 \nu^2- \nu)}   \cdot
 I_{a,c,k,\nu}(w).
\end{multline*}
\item If $c \nmid k$, $2|k$, and $c_1 \not=2$, then  we have
\begin{multline*}
\mathcal{O} \left( \frac{a}{c};q\right)=
  -  2 e^{- \frac{2 \pi i a^2 h'k_1}{c_1c}}
 \tan \left( \frac{\pi a}{c} \right)    \frac{ \omega_{h,k}^2}{    \omega_{h,k/2} } \,
 w^{-\frac{1}{2}}  \cdot   (-1)^{c_1(l+k_1)} \,
 \mathcal{O} \left(ah',\frac{lc}{c_1} ,c;q_1  \right) \\
  +
  \frac{4 \sin^2\left( \frac{\pi a }{c}  \right) \cdot   \omega_{h,k}^2}{\omega_{h,k/2} \cdot k }
\cdot  w^{\frac{1}{2}}
 \sum_{\nu \pmod k} (-1)^{\nu}\,
 e^{ - \frac{2\pi i h' \nu^2}{k} }
    \cdot
 I_{a,c,k,\nu}(w).
\end{multline*}
\item If $c \nmid k$,  $2|k$, and $c_1=2$, then  we have
\begin{multline*}
\mathcal{O} \left( \frac{a}{c};q\right)=
  -  e^{- \frac{\pi i a^2 h'k_1}{c}}
 \cdot \tan \left( \frac{\pi a}{c} \right)   \frac{\omega_{h,k}^2}{   \omega_{h,k/2} } \cdot
 w^{-\frac{1}{2}}  \cdot
 \mathcal{V} \left(\frac{a h'}{c};q_1  \right) \\
  +
  \frac{4 \sin^2\left( \frac{\pi a }{c}  \right) \cdot   \omega_{h,k}^2}{\omega_{h,k/2} \cdot k }
\cdot  w^{\frac{1}{2}}
 \sum_{\nu \pmod k} (-1)^{\nu}\,
 e^{ - \frac{2\pi i h' \nu^2}{k} }
    \cdot
 I_{a,c,k,\nu}(w).
\end{multline*}
\item  If $c \nmid k$, $2 \nmid k$, and $c_1 \not=4$, then we have
\begin{multline*}
\mathcal{O} \left( \frac{a}{c};q\right)=
-  \sqrt{2} e^{\frac{\pi i h'}{8k}    -\frac{2 \pi i h' a^2k_1}{cc_1}  }            \cdot  \tan \left( \frac{\pi a}{c} \right) \frac{\omega_{h,k}^2}{  \omega_{2h,k}} \,
 w^{-\frac{1}{2}}  \cdot
 \mathcal{U} \left(ah', \frac{l c}{c_1}, c;q_1  \right)  \\
  +    \frac{4 \sqrt{2}  \cdot \sin^2\left( \frac{\pi a }{c}  \right) \cdot   \omega_{h,k}^2}{\omega_{2h,k}\cdot k }
\cdot  w^{\frac{1}{2}}
 \sum_{ \nu  \pmod k} e^{-\frac{\pi i  h'}{k} (2 \nu^2- \nu)}   \cdot
 I_{a,c,k,\nu}(w).
   \end{multline*}
   \item  If $c \nmid k$, $2 \nmid k$, and $c_1 =4$, then we have
\begin{multline*}
\mathcal{O} \left( \frac{a}{c};q\right)=
-  e^{\frac{\pi i h'}{8k}    -\frac{2 \pi i h' a^2k_1}{cc_1}  }           \cdot  \tan \left( \frac{\pi a}{c} \right) \frac{\omega_{h,k}^2}{\sqrt{2} \cdot   \omega_{2h,k}} \cdot
 w^{-\frac{1}{2}}  \cdot
 \mathcal{V} \left(ah', \frac{l c}{c_1}, c;q_1  \right)  \\
  +    \frac{4 \sqrt{2}  \cdot \sin^2\left( \frac{\pi a }{c}  \right) \cdot   \omega_{h,k}^2}{\omega_{2h,k}\cdot k }
\cdot  w^{\frac{1}{2}}
 \sum_{ \nu  \pmod k} e^{-\frac{\pi i  h'}{k} (2 \nu^2- \nu)}   \cdot
 I_{a,c,k,\nu}(w).
   \end{multline*}
   \end{enumerate}
\end{theorem}
\begin{corollary} \label{transcorollary}
Assume  that $z \in \H$, $0 <a<c$ with $c \not= 2$.
\begin{enumerate}
\item If $c \not=4$, then  we have
 \begin{eqnarray*}
 \mathcal{O} \left( \frac{a}{c}; -\frac{1}{z}\right)= -  \sqrt{2}
  \tan \left(\frac{\pi a}{c} \right)
  \cdot
 (- i z)^{\frac{1}{2}}  \cdot
 \mathcal{U} (0,a,c;z) +    4
  \sqrt{2}  \cdot \sin^2\left( \frac{\pi a }{c}  \right)   \cdot (-i z)^{-\frac{1}{2}  }\cdot
    I_{a,c,1,0}\left(\frac{i}{z}\right).
 \end{eqnarray*}
 \item If $c  =4$, then  we have
 \begin{eqnarray*}
 \mathcal{O} \left( \frac{a}{c}; -\frac{1}{z}\right)= -
\frac{
   \tan \left(\frac{\pi a}{c} \right)  }{\sqrt{2}}  \,     (- i z)^{\frac{1}{2}}  \cdot
 \mathcal{V} (0,a,c;z) +   4
  \sqrt{2}  \cdot \sin^2\left( \frac{\pi a }{c}  \right)   \cdot (-i z)^{-\frac{1}{2}  }\cdot
    I_{a,c,1,0}  \left(\frac{i}{z}\right).
 \end{eqnarray*}
 \end{enumerate}
\end{corollary}
\begin{corollary}
\label{transcorollary2}
Assume that $\left(\begin{smallmatrix}\alpha&\beta\\ \gamma& \delta \end{smallmatrix} \right) \in \Gamma_0(c)$ with $c$ odd,  $z \in \H$, and let  $\gamma_1:= \frac{\gamma}{(c,\gamma)}$.
\begin{enumerate}
\item   If $2| \gamma$, then the holomorphic part of $\mathcal{O} \left( \frac{a}{c}; \frac{\alpha z + \beta}{\gamma z + \delta}  \right)$ is given by
\begin{eqnarray*}
-  i   \cdot e^{ \frac{2 \pi i a^2 \delta \gamma_1}{c}} \cdot
 \tan \left( \frac{\pi a}{c} \right) \cdot   \cot \left( \frac{\pi a \delta}{c}\right)   \frac{\omega_{\alpha,\gamma}^2}{ \omega_{\alpha, \gamma /2} } \,
 (-i(\gamma z+ \delta))^{\frac{1}{2}}  \cdot
 \mathcal{O} \left(\frac{a\delta}{c}; z \right).
\end{eqnarray*}
\item  If $\gamma$ is  odd,
 then the holomorphic part of $\mathcal{O} \left( \frac{a}{c}; \frac{\alpha z + \beta}{\gamma z + \delta}  \right)$ is given by
\begin{eqnarray*}
 \sqrt{2} i \cdot e^{-\frac{\pi i \delta}{8 \gamma}  + \frac{2 \pi i a^2 \delta \gamma_1}{c}}
 \tan \left( \frac{\pi a}{c} \right)  \frac{\omega_{\alpha,\gamma}^2}{ \omega_{2\alpha,\gamma} } \cdot
 (-i(\gamma z + \delta))^{\frac{1}{2}}    \cdot
 \mathcal{U} \left(-\frac{a\delta}{c};z  \right).
 \end{eqnarray*}
\end{enumerate}
\end{corollary}
\begin{proof}[Proof of Theorem \ref{Otransfo}]
We proceed similarly as in \cite{An2,B1}.
First we  rewrite (\ref{overrankgen}) as
\begin{equation} \label{Ofctrewritten} \qquad
\mathcal{O}\left (\frac{a}{c};q\right)
=
4 \sin^2 \left(\frac{\pi a}{c} \right)
\frac{\eta\left(     \frac{2}{k} (h+iw)\right)}{\eta^2\left( \frac{1}{k} (h+iw)\right)}  \sum_{n  \in \Z } (-1)^n \,
e^{\frac{2 \pi i n^2 (h+iw)}{k}}
\cdot
 H_{a,c} \left(
 \frac{2\pi i n}{k} (h+iw)
 \right)
  .
\end{equation}
Now let
\begin{eqnarray} \label{tildeO}
\widetilde{\mathcal{O}} \left( \frac{a}{c};q \right) := \frac{\eta^2  \left( \frac{1}{k} (h+iw)\right)}{4 \sin^2 \left( \frac{\pi a}{c} \right) \cdot    \eta\left( \frac{2}{k} (h+iw)\right)}
\cdot \mathcal{O} \left(\frac{a}{c};q \right).
\end{eqnarray}
Writing $n=km+\nu$ with $0 \leq \nu <k$, $m \in \Z$  gives that
$\widetilde{\mathcal{O}} \left( \frac{a}{c};q \right)$ equals
\begin{equation}\label{formel2}
\begin{split}
 \sum_{\nu=0}^{k-1} (-1)^{\nu} \, e^{\frac{2 \pi i h \nu^2}{k}}
\sum_{m \in \Z}  (-1)^{km} \ H_{a,c} \left(\frac{2 \pi i h \nu}{k} -\frac{2 \pi w}{k} (km+ \nu)\right) \cdot
e^{ -  \frac{2 \pi w}{k}(km+\nu)^2}.
\end{split}
\end{equation}
Using Poisson summation and substituting $x \mapsto kx+ \nu$ gives that the inner sum equals
\begin{eqnarray} \label{formel3a}
\frac{1}{k} \sum_{n \in \Z} \int_{\R}
 H_{a,c} \left(\frac{2 \pi i h\nu}{k} -\frac{2\pi x w}{k}\right)\cdot
 e^{  \frac{\pi i}{k}(2n+\widetilde k )(x- \nu)-  \frac{2 \pi w  x ^2}{k}}
 dx.
\end{eqnarray}
Strictly speaking for $c|k$ there may lie a pole at $x=0$. In this case we take the principal part of the integral.
Inserting  (\ref{formel3a})  into (\ref{formel2}) we  see that the summation only depends  on $\nu \pmod k$.
Moreover, by changing $\nu$ into $-\nu$,  $x$ into $-x$, and $n$ into $-(n+\widetilde k)$, we see that
the part of the sum over $n$ with $n \leq -1$ equals the part of the  sum with $n \geq 0$. Thus (\ref{formel2}) equals
\begin{equation}\label{formel3}
\frac{2   }{k}
 \sum_{\nu \pmod k}
 (-1)^{\nu} e^{\frac{2 \pi i h \nu^2}{k}}
 \sum_{n \in \N} \int_{\R}
 H_{a,c} \left(\frac{2 \pi i h\nu}{k} -\frac{2\pi x w}{k}\right)\cdot
 e^{\frac{   \pi i}{k}(2n+\widetilde k)(x- \nu)-   \frac{2 \pi w x ^2}{k}}
  dx.
\end{equation}
Next we introduce the function
\begin{eqnarray*}
S_{a,c,k}( x):=
\frac{\sinh(c_1x)}{\sinh\left(\frac{x}{k} + \frac{\pi i a}{c}\right)  \cdot     \sinh\left(\frac{x}{k} - \frac{\pi i a}{c}\right)}
\end{eqnarray*}
which is   entire  as a function of $x$. Here we need that $c \not =2$.
We rewrite  the integrand in (\ref{formel3})  as
\begin{eqnarray*}
  \frac{(-1)^{hc_1 \nu}
 e^{ \frac{  \pi i (2n+\widetilde k) (x- \nu)}{k}-  \frac{2 \pi w x ^2}{k}     }
 \cdot  S_{a,c,k}\left( \pi x w - \pi i h \nu\right)}{  4
 \sinh (\pi c_1x w)
  }.
  \end{eqnarray*}
  From this we see that the only  poles can lie in the points
$$
x_m:=\frac{im}{c_1 w} \qquad (m \in \Z).
$$
If $c |k$, then $c_1=1$; thus poles can only lie in points of the form
$
x_m=\frac{im}{w}.
$ One can easily compute that  each choice $\pm $ leads at most for one $\nu \pmod k$ to a non-zero residue, and that this $\nu$ can  be chosen as
\begin{eqnarray*}
\nu_m^{\pm} :=- h' (m \mp ak_1).
\end{eqnarray*}
If $c \nmid k$, then we can only have a nontrivial residue if
$
m \equiv \pm a k_1 \pmod{c_1}
$.
We write
$
 c_1 m \pm l
$
instead of $m$
with $m \geq \frac12 (1 \mp 1)$. We  see that to each  choice $+$ or $-$ there corresponds exactly one $\nu \pmod k$ and we can choose $\nu$ as
\begin{eqnarray*}
\nu_m^{\pm} :=   -h' \left( m \pm \frac{1}{c_1}\left(l-a k_1\right)\right).
\end{eqnarray*}
Now  shift the path of integration through the points
\begin{eqnarray*}
\omega_n:=\frac{\left(2n+\widetilde k\right)i}{4w}.
\end{eqnarray*}
Which points $x_m$ $(m \geq 0)$  we have to take into account when we use  the Residue Theorem depends on whether $c|k$ or not and on whether $k$ is even or odd.  The cases that $c_1=2,4$ require special care.
\begin{itemize}
\item
If $c|k$ and $k$ is even, then we have to take those $x_m$ into account  for which  $2m \leq n$. The poles on the   path of integration are  $x_0$ and $\omega_n$.
\item
If $c|k$  and $k$ is  odd, then we have to take those $x_m$ into account  for which  $2m \leq n$.
The point $x_0$  is the only pole on the path of integration
.
\item
 If $c \nmid k$, $k$ is even, and $c_1 \not=2$, then
 there is no pole on the path of integration. Moreover in this case
 we have to take those $x_m$ into account  for which  $n \geq 2m + \frac{1}{2} (1 \pm t)$.
 \item
If $c \nmid k$, $k$ is even, and $c_1=2$, then the only   pole on the path of integartion lies in $\omega_n$.
We have to take those $x_m$ into account for which $ n \geq 2m \pm 1$.
\item
If $c \nmid k$, $k$ is odd, and $c_1 \not = 4$, then there is no pole on the path of integration. Moreover we have to take those  $x_m$ into account for which
$ n \geq 2m \pm s$.
\item
If $c \nmid k$, $k$ is odd, and
 $c_1 =4$, then there lies a pole in $\omega_n$ and we have to take those $x_m$ into account for which $n \geq 2m \pm s + \frac{1}{2} (-1 \pm 1)$.
 \end{itemize}
 In the following we denote the residues of the integrand by $\lambda_{n,m}^{\pm}$.  It is not hard to compute
 \begin{eqnarray*}
\lambda_{n,m}^{\pm}=    \pm
\frac{i k
 }{ 4 \pi w \sin \left(  \frac{2 \pi a}{c} \right)}   \,
e^{\frac{ \pi i}{k} (2n+\widetilde k)\left(x_m- \nu_m^{\pm}\right) - \frac{2 \pi w x_m^2}{k}}.
\end{eqnarray*}
>From this one directly sees that
\begin{eqnarray*}
\lambda_{n+1,m}^{\pm}
=\exp \left( \frac{2 \pi i}{k}\left(x_m - \nu_{m}^{\pm} \right)\right) \cdot
\lambda_{n,m}^{\pm}.
\end{eqnarray*}
Shifting the path of integration through $\omega_n$, we   obtain by the Residue Theorem
\begin{eqnarray*}
 \widetilde O  \left( \frac{a}{c};q\right)=
\sum_{1}+\sum_{2},
\end{eqnarray*}
where
\begin{eqnarray*}
 \sum_2:=
 \frac{2  }{k}
  \sum_{\nu \pmod k}(-1)^{\nu} \ e^{\frac{2 \pi i h \nu^2}{k}}
 \sum_{n \in \N} \int_{-\infty+\omega_n}^{\infty+\omega_n}
 H_{a,c} \left(\frac{2 \pi i h\nu}{k} -\frac{2 \pi x w}{k}\right)\cdot
 e^{\frac{\pi i (2n+\widetilde k )(x- \nu)}{k}-   \frac{2 \pi x ^2w}{k}} \
  dx.
\end{eqnarray*}
For the definition of $\sum_{1}$   we have to distinguish several  cases.   We set $r_0:=\frac{1}{2}$ and $r_m:=1$ for $m \in \N$.
If $c|k$ and $k$ is even, then
$$
\sum_{1}:=
\frac{4 \pi i  }{k}
  \sum_{m\geq 0 \atop \epsilon \in \{\pm \}}
  r_m   (-1)^{\nu_m^{\epsilon}}\ e^{\frac{2 \pi i h(\nu_m^{\epsilon})^2}{k}}
  \left(
  \frac{\lambda_{2m+1,m}^{\epsilon}}{1 - \exp \left(\frac{2 \pi i }{k} \left(x_m- \nu_m^{\epsilon} \right) \right)}
  + \frac{1}{2}  \lambda_{2m,m}^{\epsilon}
  \right).
   $$
   An easy calculation shows that this equals
   \begin{eqnarray} \label{1sum1}
  (-1)^{k_1} i
  e^{- \frac{2 \pi i h' a^2 k_1}{c}}  \frac{\sin \left( \frac{2 \pi h' a}{c} \right)}{w \cdot \sin \left(\frac{2 \pi a}{c}  \right)}
  \sum_{n  \in \Z }
\frac{(-1)^n q_1^{n^2+n}}
{1 - 2 q_1^n \cos  \left(\frac{2 \pi a h'}{c} \right)+q_1^{2n}} .
   \end{eqnarray}
   If $c|k$  and  $2 \nmid k$, then
   \begin{eqnarray} \label{resodd}
\sum_{1}:=
\frac{4 \pi i  }{k}
  \sum_{m\geq 0 \atop \epsilon \in \{\pm \}}
  r_m   (-1)^{\nu_m^{\epsilon}}\ e^{\frac{2 \pi i h\left(\nu_m^{\epsilon}\right)^2}{k}}
  \frac{\lambda_{2m,m}^{\epsilon}}{1 - \exp \left(\frac{2 \pi i }{k} \left(x_m- \nu_m^{\epsilon} \right) \right)}
   .
   \end{eqnarray}
   We assume without loss of generality that $h'$ is even. Then we can show that  (\ref{resodd}) equals
       \begin{eqnarray} \label{2sum1}
\frac{i \sin \left( \frac{ \pi h' a}{c} \right)}{w \cdot \sin \left(\frac{2 \pi a }{c}  \right)}\,    e^{- \frac{2 \pi i h' a^2  k_1}{c}}
  \sum_{n  \in \Z }
\frac{\left( 1+ q_1^n  \right) q_1^{n^2+\frac{n}{2}}}
{1 - 2 q_1^n \cos  \left(\frac{2 \pi a h'}{c} \right)+q_1^{2n}}.
   \end{eqnarray}
   If $c \nmid k$, $2 |k$, and $c_1 \not=2$, then
   \begin{eqnarray*}
   \sum_{1}:=
\frac{4 \pi i  }{k}
  \sum_{m\geq 0 \atop \epsilon \in \{\pm \}}
   (-1)^{\nu_m^{\epsilon}}\ e^{\frac{2 \pi i h(\nu_m^{\epsilon})^2}{k}}
  \frac{\lambda_{2m+\frac{1}{2} (1\pm t),m}^{\epsilon}}{1 - \exp \left(\frac{2 \pi i }{k} \left(x_m- \nu_m^{\epsilon} \right) \right)} .
  \end{eqnarray*}
  It can be calculated that this equals
  \begin{eqnarray*} \label{3sum1}
   -  \frac{1}{w \sin \left( \frac{2 \pi a}{c} \right)} \,
q_1^{\frac{t l}{2 c_1}  - \frac{l^2}{c_1^2}+\frac{l}{2c_1}} \cdot
e^{\frac{\pi i h'a}{c} \left(\frac{4l}{c_1} -\frac{2ak_1}{c_1}- t -1\right)}
(-1)^{c_1(l+k_1)}
\sum _{m\in \Z}  (-1)^m
\frac{q_1^{\frac{m}{2}(2m+1) + \frac{m t}{2}}}{1- e^{-\frac{2 \pi i a h'}{c}} \cdot q_1^{m+\frac{l}{c_1}}} .
%-
%\sum _{m=1}^{\infty}
%(-1)^m
%\frac{e^{\frac{\pi i ah'}{c}}\cdot q_1^{\frac{m}{2}(2m+1) - \frac{mt}{2}- \frac{l}{2c_1}}}{1- e^{\frac{2 \pi i a h'}{c}} \cdot q_1^{m-\frac{l}{c_1}}}
%\right).
\end{eqnarray*}
If $c \nmid k$, $2|k$, and $c_1=2$, then
\begin{multline*}
\sum_{1}:=
\frac{4 \pi i  }{k}\left(
  \sum_{m\geq 0}
    (-1)^{\nu_m^{+}}\ e^{\frac{2 \pi i h(\nu_m^{+})^2}{k}}
  \left(
  \frac{\lambda_{2m+2,m}^{+}}{1 - \exp \left(\frac{2 \pi i }{k} \left(x_m- \nu_m^{+} \right) \right)}
  + \frac{1}{2}  \lambda_{2m+1,m}^{+}
  \right)
  \right.
  \\
  \left.
  +
   \sum_{m\geq 1}
    (-1)^{\nu_m^{-}}\ e^{\frac{2 \pi i h(\nu_m^{-})^2}{k}}
  \left(
  \frac{\lambda_{2m,m}^{-}}{1 - \exp \left(\frac{2 \pi i }{k} \left(x_m- \nu_m^{-} \right) \right)}
  + \frac{1}{2}  \lambda_{2m-1,m}^{-}
  \right)
  \right)
  .
   \end{multline*}
   One can show that this equals.
   \begin{eqnarray} \label{4sum1}
   -  \frac{q_1^{\frac{1}{4}}}{2w \sin \left( \frac{2 \pi a}{c} \right)} \,  e^{-\frac{\pi i a^2h'k_1}{c}}
\sum _{m\in \Z}
\frac{  q_1^{m^2+ m     }    \left( 1+ e^{- \frac{2 \pi i  h'a}{c}}   q_1^{m+\frac{1}{2}}\right) }{1- e^{-\frac{2 \pi i a h'}{c}} \cdot q_1^{m+\frac{1}{2}}}   .
%\left.
%-
%\sum _{m=1}^{\infty}
%\frac{  q_1^{m^2- m }    \left( 1+ e^{ \frac{2 \pi i  h'a}{c}}   q_1^{m-\frac{1}{2}}\right) }{1- e^{\frac{2 \pi i a h'}{c}} \cdot q_1^{m-\frac{1}{2}}}
%\right).
\end{eqnarray}
  If $c \nmid k$, $k$ is odd, and $c_1 \not=4$, then
   \begin{multline*}
   \sum_{1}:=
\frac{4 \pi i  }{k}  \left(
  \sum_{m\geq 0}
   (-1)^{\nu_m^{+}}\ e^{\frac{2 \pi i h(\nu_m^{+})^2}{k}}
  \frac{\lambda_{2m+s,m}^{+}}{1 - \exp \left(\frac{2 \pi i }{k} \left(x_m- \nu_m^{+} \right) \right)}
  \right.  \\
  \left.
  +
   \sum_{m\geq 1}
   (-1)^{\nu_m^{-}}\ e^{\frac{2 \pi i h(\nu_m^{-})^2}{k}}
  \frac{\lambda_{2m-s,m}^{-}}{1 - \exp \left(\frac{2 \pi i }{k} \left(x_m- \nu_m^{-} \right) \right)}  \right)
  .
  \end{multline*}
  Without loss of generality we may assume that  $h'$ is even. With this assumption we can compute that $\sum_1$ equals
    \begin{eqnarray*} \label{5sum1}
   -  \frac{1}{w \sin \left( \frac{2 \pi a}{c} \right)}  \,
q_1^{\frac{s l}{c_1}  - \frac{l^2}{c_1^2}+\frac{l}{2c_1}}  \cdot
e^{\frac{ \pi i h'a}{c} \left(\frac{4l}{c_1} -\frac{2ak_1}{c_1}-2s-1\right)}
\sum _{m \in \Z}
\frac{q_1^{\frac{m}{2}(2m+1) + ms}}{1- e^{-\frac{2 \pi i a h'}{c}} \cdot q_1^{m+\frac{l}{c_1}}}.
%-
%\sum _{m=1}^{\infty}
%\frac{e^{\frac{\pi i ah'}{c}}\cdot q_1^{\frac{m}{2}(2m+1) - ms- \frac{l}{2c_1}}}{1- e^{\frac{2 \pi i a h'}{c}} \cdot q_1^{m-\frac{l}{c_1}}}
%\right).
\end{eqnarray*}
  If $c \nmid k$, $k$ is odd, and $c_1 =4$, then
   \begin{multline*}
   \sum_{1}:=
\frac{4 \pi i  }{k}
  \left(
  \sum_{m=0} ^{\infty}
   (-1)^{\nu_m^{+}}\ e^{\frac{2 \pi i h(\nu_m^{+})^2}{k}}
   \left(
  \frac{\lambda_{2m+s+1,m}^{+}}{1 - \exp \left(\frac{2 \pi i }{k} \left(x_m- \nu_m^{+} \right) \right)}
  + \frac{1}{2}  \lambda_{2m+s,m}  \right)
  \right.  \\
  \left.
  +
   \sum_{m = 1}^{\infty}
   (-1)^{\nu_m^{-}}\ e^{\frac{2 \pi i h(\nu_m^{-})^2}{k}}   \left(
  \frac{\lambda_{2m-s,m}^{-}}{1 - \exp \left(\frac{2 \pi i }{k} \left(x_m- \nu_m^{-} \right) \right)}
    + \frac{1}{2}  \lambda_{2m-s-1,m} \right)
  \right)
  .
  \end{multline*}
    One can show that this equals
      \begin{equation}  \label{6sum1}
   -  \frac{1}{2 w \sin \left( \frac{2 \pi a}{c} \right)}
q_1^{\frac{s l}{c_1}  - \frac{l^2}{c_1^2}+\frac{l}{2c_1}}
e^{\frac{\pi i h'a}{c} \left(\frac{4l}{c_1} -\frac{ak_1}{2c_1}-2s \right)}
\sum _{m \in \Z}
\frac{ q_1^{\frac{m}{2}(2m+1) + ms  }
 \left(
1+e^{-\frac{2 \pi i a h'}{c}}  q_1^{m+\frac{l}{c_1}} \right)
}{1- e^{-\frac{2 \pi i a h'}{c}} \cdot q_1^{m+\frac{l}{c_1}}}.
%\right. \\
%\left.
%-
%\sum _{m=1}^{\infty}
%\frac{e^{-\frac{\pi i ah'}{c}} q_1^{\frac{m}{2}(2m-1) - ms + \frac{l}{2c_1}}     \left(
%1+e^{\frac{2 \pi i a h'}{c}}  q_1^{m-\frac{l}{c_1}}
%\right)}{1-  e^{\frac{2 \pi i a h'}{c}} \cdot q_1^{m-\frac{l}{c_1}}}
%\right).
\end{equation}
We next turn to the computation of $\sum_2$.
If there is   a pole in $\omega_n$ we take  the principal part of the integral.
With the same argument as before we can change the sum over $\N$ into a sum over $\Z$. Moreover we make   the translation
$x \mapsto x+ \omega_n$ and write $n=2 p + \delta$ with $ p \in \Z$ and $ \delta \in \{ 0, - 1\}$.
This gives
\begin{multline*}
 \sum_2 =
 \frac{ 1 }{k}
   \sum_{\nu \pmod k}
  (-1)^{\nu} \ e^{\frac{2\pi i h \nu^2}{k} }
 \sum_{\substack{p \in \Z   \\ \delta \in \{ 0, -1\}  }}
 e^{-\frac{\pi i}{k}  (4p+ 2\delta+ \widetilde k) \nu  -\frac{\pi }{8kw} (4p + 2 \delta + \widetilde k)^2   }
  \\
   \int_{-\infty}^{\infty}
 H_{a,c} \left(\frac{2\pi i h \nu}{k}      -\frac{2\pi wx}{k} - \frac{\pi i}{2k}  \left(4p+ 2\delta + \widetilde k\right) \right)
  \cdot     e^{-\frac{2 \pi w x^2}{k}} \ dx.
\end{multline*}
Next  we change $\nu$ into $- h'( \nu+p)$  and   distinguish whether $k$ is even or odd.

 If $k$ is even, then  we have, since $h'$ is odd,
\begin{multline*}
 \sum_2 =
 \frac{ 1 }{k}
   \sum_{\substack{\nu \pmod k\\ \delta \in \{0,-1\}}}
  (-1)^{\nu} \ e^{\frac{\pi i h'}{2k}\left( -\delta^2  + 4 \nu (\delta- \nu ) \right)}
 \sum_{p \in \Z}  (-1)^p    \,
 q_1^{\frac{(4p+2 \delta)^2}{16}} \\
   \int_{-\infty}^{\infty}
 H_{a,c} \left(\frac{2\pi i \nu}{k}      -\frac{2\pi wx}{k} - \frac{\pi i \delta}{k} \right)
  \cdot     e^{-\frac{2 \pi w x^2}{k}} \ dx.
\end{multline*}
Now the integral is independent of $p$. Moreover the sum over $p$ vanishes for $\delta=-1$ since the $p$th and the $(-p+1)$th term cancel.   For $\delta=0$  the sum over $p$ equals
\begin{eqnarray}\label{sump}
\sum_{p \in \Z}
(-1)^p \cdot q_1^{p^2}
= \frac{\eta^2 \left(\frac{1}{k} \left(h' + \frac{i}{w} \right)  \right)}{  \eta \left(\frac{2}{k} \left(h' + \frac{i}{w} \right)  \right)}
.
\end{eqnarray}
Thus
\begin{eqnarray}\label{inteven}
\sum_2=
 \frac{\eta^2 \left(\frac{1}{k} \left(h' + \frac{i}{w} \right)  \right)}{ k \cdot  \eta \left(\frac{2}{k} \left(h' + \frac{i}{w} \right)  \right)}
\sum_{\nu \pmod k}
(-1)^{\nu} \,
e^{-\frac{2\pi i h' \nu^2}{k}}
 \cdot
I_{a,c,k,\nu}(w).
\end{eqnarray}
 If $k$ is odd, then  we may without loss of generality assume that   $h'$ is even.
 In this case we obtain
 \begin{multline*}
 \sum_2 =
 \frac{ 1 }{k}  e^{-  \frac{\pi i h'}{8k} }
  \sum_{\substack{\nu \pmod k\\ \delta \in \{0,-1\}}}
  e^{\frac{\pi i h'}{k}\left( - 2 \nu^2 +  \nu (2 \delta+1)  \right)}
 \sum_{p \in \Z}
 q_1^{\frac{(4p+2 \delta+1 )^2}{16}} \\
   \int_{-\infty}^{\infty}
 H_{a,c} \left(\frac{2\pi i \nu}{k}      -\frac{2\pi wx}{k} - \frac{\pi i (2\delta+1)}{2k} \right)
  \cdot     e^{-\frac{2 \pi w x^2}{k}} \ dx.
\end{multline*}
Making the substitutions $p \mapsto -p$, $x \mapsto -x$, and $\nu \mapsto - \nu$,  one can easily see that the contribution for $\delta=0$ and for $\delta=-1$ coincide.
Moreover the sum over $p$ equals
\begin{eqnarray}\label{sump1}
\sum_{p \in \Z}  q_1^{\frac{(4p+1)^2}{16}}
=
\frac{\eta^2 \left(\frac{1}{k} \left(h' + \frac{i}{w} \right)  \right)}{ \eta \left(\frac{1}{2k} \left(h' + \frac{i}{w} \right)  \right)}
.
\end{eqnarray}
Thus
\begin{eqnarray}\label{intodd}
\sum_2=
\frac{2}{k}   \cdot e^{-\frac{\pi i h'}{8k}}\,
 \frac{\eta^2 \left(\frac{1}{k} \left(h' + \frac{i}{w} \right)  \right)}{  \eta \left(\frac{1}{2k} \left(h' + \frac{i}{w} \right)  \right)}\,
\sum_{\nu \pmod k}
e^{\frac{ \pi i h' }{k}\left( -2 \nu^2 +\nu \right)} \cdot
I_{a,c,k,\nu}(w).
\end{eqnarray}
To finish the proof of
Theorem \ref{Otransfo}, we require  the well-known transformation law of Dedekind's $\eta$-function.
\begin{eqnarray} \label{etatrans}
\eta \left( \frac{1}{k}(h+iw)\right)
= e^{ \frac{\pi i }{12k} (h-h') } \cdot
\omega_{h,k}^{-1} \cdot w^{-\frac{1}{2}}  \cdot
\eta \left( \frac{1}{k}\left(h'+\frac{i}{w}\right)\right).
\end{eqnarray}
This implies that for $k$ even, we have
\begin{eqnarray} \label{etaeven}
\eta \left( \frac{2}{k}(h+iw)\right)
= e^{ \frac{\pi i }{6k} (h-\widetilde h) } \cdot
\omega_{h,k/2}^{-1} \cdot w^{-\frac{1}{2}}  \cdot
\eta \left( \frac{2}{k}\left(\widetilde h+\frac{i}{w}\right)\right),
\end{eqnarray}
where $h \widetilde h \equiv -1 \pmod {k/2}$.
Moreover if $k$ is odd, (\ref{etatrans}) implies that
\begin{eqnarray} \label{etaodd}
\eta \left( \frac{2}{k}(h+iw)\right)
= e^{ \frac{\pi i }{12k} (2h-(2h)') } \cdot
\omega_{2h,k}^{-1} \cdot (2w)^{-\frac{1}{2}}  \cdot
\eta \left( \frac{1}{k}\left((2h)'+\frac{i}{2w}\right)\right).
\end{eqnarray}
Combining (\ref{1sum1}), (\ref{2sum1}), (\ref{3sum1}),  (\ref{6sum1}),
(\ref{inteven}), (\ref{intodd}),  (\ref{etatrans}), (\ref{etaeven}), and (\ref{etaodd}) gives (after a lengthy
but straightforward calculation) the theorem.
  \end{proof}
  %%%%%%%%%%%%%%%%%%%%%%%%%%%%%%%
  %%%%%%%%%%%%%%%%%%%%%%%%%%%%%%%
  %%%%%%%%%%%%%%%%%%%%%%%%%%%%%
  \section{Construction of the weak Maass forms} \label{SectionMaass}
  In this section we prove the first part of Theorem \ref{maassform}.
  First we interpret the Mordell type integral occurring in Corollary \ref{transcorollary}  as an integral of theta functions.
  For this let
  \begin{eqnarray} \label{thetaintegral}
  I_{z}:=  \frac{4\sqrt{2}\sin^2 \left(\frac{\pi a}{c} \right) }{(-i z)}
  \int_{\R} e^{-\frac{2 \pi i x^2}{z}} \cdot
  H_{a,c} \left(
  \frac{2 \pi i x}{z}   + \frac{ \pi i}{2}
  \right) dx.
  \end{eqnarray}
  \begin{lemma} \label{thetalemma}
  We have
  \begin{eqnarray*}
  I_{z} =  \frac{\pi  \tan \left( \frac{\pi a}{c}\right)}{4c}
  \int_{0}^{\infty}
  \frac{\Theta_{a,c}\left( iu  \right)}{\sqrt{-i(iu+z)}} \,du.
  \end{eqnarray*}
  \end{lemma}
  \begin{proof}
  We modify a proof  of \cite{BO2,Zw}.
  By analytic continuation it is enough to show the claim for $z=it$ with $t>0$.
  Making the change of variables $x \mapsto \frac{x}{t}$, we find that
   \begin{eqnarray}  \label{Itau}
  I_{it }=   4  \sqrt{2}\sin^2 \left(\frac{\pi a}{c} \right)
  \int_{\R} e^{-2 \pi t  x^2}  \cdot
  H_{a,c} \left(
  2 \pi  x   + \frac{ \pi i}{2}
  \right) dx.
  \end{eqnarray}
  We next rewrite  $H_{a,c} \left(
  2 \pi  x   + \frac{ \pi i}{2}
  \right)$ using the Mittag-Leffler theory of partial fraction decomposition. This easily gives  that
  \begin{equation} \label{ML}
  H_{a,c} \left(
  2 \pi x   + \frac{ \pi i}{2}
  \right)
  =
  \frac{i}{4 \pi \sin \left( \frac{2 \pi a}{c}\right)}
  \sum_{m \in \Z}
  \left(
  \frac{1}{x- i \left( m - \frac{a}{c} -\frac{1}{4}\right)}
  -      \frac{1}{x- i \left( m + \frac{a}{c} -\frac{1}{4}\right)}
  \right).
  \end{equation}
  We plug (\ref{ML}) back into (\ref{Itau}) and interchange summation and integration. For this we introduce the extra
  summands $ \frac{1}{ i \left( m - \frac{a}{c} - \frac{1}{4}\right)}$ and  $ \frac{1}{ i \left( m + \frac{a}{c} - \frac{1}{4}\right)}$ which enforce absolute convergence and cancel when we integrate.  This gives
  \begin{eqnarray*}
  I_{it}
  = \frac{i \tan\left( \frac{\pi a}{c} \right) }{\sqrt{2}}
  \sum_{m \in \Z}
  \int_{\R}
  e^{- 2 \pi t x^2}
   \left(
  \frac{1}{x- i \left( m - \frac{a}{c} -\frac{1}{4}\right)}
  -      \frac{1}{x- i \left( m + \frac{a}{c} -\frac{1}{4}\right)}
  \right) dx.
  \end{eqnarray*}
  Next use  that  for all  $s \in \R \setminus\{ 0\}$, we have the identity
 \begin{eqnarray*}
  \int_{- \infty}^{\infty}\frac{e^{- \pi t  x^2}}{x- is}\ dx
  = \pi i s \int_{0}^{\infty}\frac{e^{- \pi u  s^2}}{\sqrt{u+t}}\ du
 \end{eqnarray*}
 (this follows since both sides are solutions of
 $\left(- \frac{\partial}{\partial t}
 + \pi s^2\right) f(t) = \frac{\pi i s}{\sqrt{t}}\, f(t)$ and have the same limit
 $0$ as $ t \mapsto \infty$ and hence are equal).
 Again interchanging summation and integration and making the change of variables   $u \mapsto \frac{u}{2}$
 gives
 \begin{multline*}
  I_{it}
  = - \pi \tan \left( \frac{\pi a}{c} \right)
   \int_{0}^{\infty}   \frac{1}{\sqrt{u+t}}  \\
   \sum_{m \in \Z}
   \left(
   \left( m - \frac{a}{c}-\frac{1}{4}
   \right) \cdot e^{-  2 \pi u \left( m - \frac{a}{c} -\frac{1}{4}   \right)^2}
   -
     \left( m + \frac{a}{c}-\frac{1}{4}
   \right) \cdot e^{-  2 \pi u \left( m + \frac{a}{c} -\frac{1}{4}   \right)^2}
   \right) \, du.
     \end{multline*}
     From this the claim can be easily deduced.
\end{proof}
\begin{lemma}   \label{transformationnonhol}
 For $z \in \H$, we have
\begin{eqnarray*}
J\left(\frac{a}{c};z+1\right)&=&J\left( \frac{a}{c};z\right),\\
\frac{1}{\sqrt{-i z}} \cdot J\left(\frac{a}{c};  - \frac{1}{z} \right)&=&    I_z  +
\frac{
\pi i  \, \tan \left(\frac{\pi a}{c}  \right)}{4c}  \,
\int_{- \bar z}^{i \infty}
\frac{\Theta_{a,c} (\tau)}{\sqrt{-i(\tau+z)}} \, d \tau.
\end{eqnarray*}
\end{lemma}
\begin{proof}
We only show the lemma in the case   that $c$ is odd, the case $c$ even is shown similarly.
The first claim follows from the fact that
$
\Theta_{a,c} \left(-\frac{1}{\tau} \right)
$ is invariant under $\tau \mapsto \tau+1$. Indeed Shimura's work \cite{Sh} implies  that
\begin{equation}  \label{Shimura}
(- i 4c \tau)^{-\frac{3}{2}}    \cdot
\Theta_{a,c}  \left( - \frac{1}{\tau}  \right)
= - i  (2c)^{-\frac{1}{2}}  \sum_{k \pmod{2c}}
\exp \left(
\frac{2\pi i k(4a+c)}{2c}
\right)
\cdot
\Theta(k,2c;4 c \tau).
\end{equation}
To prove the second transformation law  we
directly compute
\begin{displaymath}
\frac{1}{\sqrt{-i  z}} \cdot    J_{a,c} \left(
-\frac{1}{z}\right)\\
=  -
\frac{\pi i \tan \left(\frac{\pi a}{c}  \right)}{4c    \cdot   \sqrt{-i  z}}
\int_{\frac{1}{\bar z}}^{i \infty}
\frac{(-i  \tau)^{-\frac{3}{2}}\cdot  \Theta_{a,c}\left( -
\frac{1}{ \tau}\right)} {\sqrt{-i\left(\tau -
\frac{1}{ z}\right)}} \ d\tau.
\end{displaymath}
Making the change of variable $\tau \mapsto - \frac{1}{
\tau}$  now easily gives the claim.
\end{proof}
\begin{proof}[Proof of Theorem \ref{maassform} (1)]
Again we assume that $c$ is odd. By \cite{Sh} Proposition 2.1, the functions $\Theta(k,2c;\tau)$ are  cusp forms  for $\Gamma(4c)$. Thus    the functions
$\Theta(k,2c;4c\tau)$ are  cusp forms for $\Gamma_1(16c^2)$.
Using Theorem \ref{Otransfo} one can conclude that  also $\mathcal{M}\left(\frac{a}{c};z\right)$ transforms correctly under  $\Gamma_1(16c^2)$.
That it is an eigenfunction under $\Delta_{\frac12}$ follows as in \cite{BO2} page 21.
\end{proof}
%%%%%%%%%%%%%
%%%%%%%%%%%%
%%%%%%%%%%%%%
\section{The case $u=-1$}  \label{section-1}
Here we consider the case $u=-1$.   We assume the same notation as in Section \ref{SectionTrans}.
Equation   (\ref{overrankgen}) gives
\begin{eqnarray*}
\mathcal{O}(-1;q) = 4 \frac{(-q)_{\infty}}{(q)_{\infty}}
\sum_{n \in \Z} (-1)^n \frac{q^{n^2 +n }}{\left(1+q^n \right)^2}.
\end{eqnarray*}
This function is  more complicated than $\mathcal{O}\left( \frac{a}{c};q\right)$ with $c \not=2$ since double poles occur.
To overcome this problem,   we first   prove a transformation law  for the function
\begin{eqnarray*}
\mathcal{O}_r(q)
 :=  4 \frac{(-q)_{\infty}}{(q)_{\infty}}
\sum_{n \in \Z} (-1)^{n+1} \frac{q^{n^2}}{\left(1+e^{2 \pi i r }\, q^n \right)}.
\end{eqnarray*}
This function is related to $\mathcal{O}(-1;q)$ by
\begin{eqnarray*}
\frac{1}{2 \pi i} \left[
\frac{\partial}{\partial r} \mathcal{O}_r(q)
 \right]_{r=0}= \mathcal{O}(-1;q).
\end{eqnarray*}
To state the transformation law for $\mathcal{O}_r(q)$ we additionally need the functions
\begin{eqnarray*}
\mathcal{U}_r(q)&:=& e^{  \pi i r}
\frac{\eta(z)}{\eta^2(2z)}
\sum_{\substack{m \in \Z\\ m \text{ odd} } }
\frac{q^{\frac12( m^2+m) }}{1-e^{ 2 \pi ir} q^m},\\
I_{k,\nu,r}^{\pm}(w)&:=& \int_{\R}
\frac{e^{-\frac{2 \pi wx^2}{k} }}{1+ e^{\pm \frac{2 \pi i \nu}{k} \mp \frac{\pi i \widetilde{k}}{2k} + 2 \pi i r - \frac{2\pi wx}{k}}}\, dx.
\end{eqnarray*}
\begin{theorem}   \label{transformation-1}
Assume that  $r \in \R$ with $|r|$ sufficiently small.
\begin{enumerate}
\item If $k$ is odd, then
\begin{multline*}
\mathcal{O}_r(q) = - 2 \sqrt{2} i w^{ -\frac{1}{2}} \frac{\omega_{h,k}^2}{\omega_{2h,k}}\,
e^{ \frac{\pi i h'}{8k}+ \frac{2 \pi k r^2}{w}}\, \mathcal{U}_{\frac{ir}{w}} \left(q_1^{ \frac12} \right)\\
- 2 \sqrt{2} w^{\frac12 }   \frac{\omega_{h,k}^2}{\omega_{2h,k}}\, \frac{1}{k}
\sum_{ \substack{ \nu \pmod k \\ \pm}}
e^{\frac{2 \pi i h'}{k}  \left(- \nu^2 + \frac{\nu}{2} \right)
}
I_ {k,\nu,r}^{\pm }(w).
\end{multline*}
\item If $k$ is even, then
\begin{eqnarray*}
\mathcal{O}_r(q) =
- i \frac{\omega_{h,k}^2}{\omega_{h,k/2 }}\ w^{ -\frac12} e^{ \frac{2 \pi kr^2}{w}}
\mathcal{O}_{\frac{r}{iw} } (q_1)
- \frac{4 \omega_{h,k}^2w^{\frac12 }}{k\omega_{h,k/2}}
\sum_{\nu \pmod k } (-1)^{\nu}
e^{ - \frac{2 \pi i h' \nu^2}{k}} I_ { k,\nu,r}^{+}(w).
\end{eqnarray*}
\end{enumerate}
\end{theorem}
\begin{proof}
We proceed similarly as in the proof of Theorem \ref{Otransfo} and therefore we  skip most of the  details.
 Moreover, we only show the claim for $k$ odd, the case $k$ even is treated similarly.
 Define
\begin{eqnarray*}
\widetilde{\mathcal{O}}_r(q)
 :=
\sum_{n \in \Z} (-1)^{n} \frac{q^{n^2}}{\left(1+e^{2 \pi i r }\, q^n \right)}.
\end{eqnarray*}
Changing $n$ into $\nu + km$ with $\nu$ running modulo $k$ and $m \in \Z$, we obtain, using Poisson summation and making  a change of variables,
\begin{eqnarray*}
\widetilde{\mathcal{O}}_r(q)
= \frac{1}{k}
\sum_{\nu \pmod k}
(-1)^{\nu}\, e^{\frac{2 \pi i h \nu^2}{k}} \sum_{\substack{ n \in \N_0\\\pm}}
\int_{\R}
\frac{e^{- \frac{2 \pi w x^2}{k}  + \frac{\pi i}{k} ( 2 n + 1) (x- \nu) }}{1+ e^{2 \pi i r \pm \frac{2 \pi i h \nu}{k}  \mp \frac{2 \pi w}{k}x}}\, dx.
\end{eqnarray*}
%If $n$ is even, then the summand $n=0$ has to be counted with multiplicity $\frac{1}{2}$.
One can show that  poles of the integrand only lie in points
\begin{eqnarray*}
x_m : = \frac{i}{w} \left(m + \frac{1}{2}  \pm k r \right)
\end{eqnarray*}
and we have nontrivial residues at most for one $\nu \pmod k$ which can can chose as
\begin{displaymath}
\nu_m:=
-\frac{1+k}{2} (2m + 1) h'.
\end{displaymath}
Using the Residue Theorem, we shift the path of integration through
\begin{eqnarray*}
\omega_n : = \frac{2n + 1}{4w}.
\end{eqnarray*}
There are  no poles  on the real axis or in $\omega_n$.
Moreover we have to take those $x_m$ with $m \geq 0$ into account that satisfy $n \geq 2m+1$.
We denote by $\lambda_{n,m}^{\pm}$ the residue of each summand and   compute
\begin{eqnarray*}
\lambda_{n,m}^{\pm}  = \pm \frac{k}{2 \pi w} \, e^{ - \frac{2 \pi  wx_m^2}{k}  + \frac{\pi i}{k}(2n+1)(x_m- \nu_m) }.
\end{eqnarray*}
Thus
\begin{eqnarray*}
\lambda_{n+1,m}^{\pm} = e^{\frac{2\pi i}{k} (x_m-\nu_m)} \lambda_{n,m}^{\pm}
\end{eqnarray*}
which gives
\begin{eqnarray*}
\widetilde{\mathcal{O}}_w(q)  = \sum_1 + \sum_2,
\end{eqnarray*}
where
\begin{eqnarray*}
\sum_1&=& \frac{2 \pi i}{k}
  (-1)^{\nu_m}  e^{ \frac{2\pi i h\nu_m^2}{k}}
\sum_{ \substack{m \geq 0\\ \pm }}
\frac{\lambda_{2m+1,m}^{\pm}}{1-e^{ \frac{2\pi i}{k} (x_m-\nu_m)}}, \\
\sum_2&=&
\frac{1}{k}
\sum_{\nu \pmod k}
(-1)^{\nu}\, e^{\frac{2 \pi i h \nu^2}{k}} \sum_{\substack{ n \in \N_0\\\pm}}
\int_{\R+\omega_n}
\frac{e^{- \frac{2 \pi w x^2}{k}  + \frac{\pi i}{k} ( 2 n + 1) (x- \nu) }}{1+ e^{2 \pi i r \pm \frac{2 \pi i h \nu}{k}  \mp \frac{2 \pi w}{k}x}}\, dx.
\end{eqnarray*}
A lengthy calculation gives that
\begin{eqnarray*}
\sum_1 =
\frac{i}{w}\,
e^{\frac{2 \pi k r^2}{w}- \frac{\pi r}{w}}\sum_{\substack{m \in \Z\\ m \text{ odd}}}
\frac{q_1^{ \frac{1}{4} \left(m^2+m  \right)}}{1- e^{-\frac{2\pi r}{w}}q_1^{\frac12 m}}.
\end{eqnarray*}
To compute $\sum_2$, we change the sum over $n$ back  into a sum over $\Z$ and change $x \mapsto  x+\omega_n$.
This gives
\begin{eqnarray*}
\sum_2=\frac{1}{k} \sum_{ \nu \pmod k} (-1)^{\nu} \, e^{\frac{2 \pi i h \nu^2}{k}}
\sum_{n \in \Z} e^{-\frac{\pi}{8kw} (2n+1)^2 -\frac{\pi i}{k}(2n+1) \nu}
\int_{\R}
\frac{e^{-\frac{2\pi wx^2}{k} }}{1+e^{ -\frac{\pi i}{2k}(2n+1) + \frac{2 \pi i h \nu}{k} - \frac{2 \pi wx}{k} + 2 \pi ir}}\, dx.
\end{eqnarray*}
Changing $n \mapsto 2 p + \delta$ (with $p \in \Z$ and $\delta \in \{0,-1\}$) and $\nu \mapsto - h' (\nu+p)$ gives
\begin{multline*}
\sum_2 = \frac{1}{k} \sum_{\nu,p, \delta}
e^{ \frac{2 \pi i h'}{k} \left( - \nu^2 -\frac{1}{16}(2 \delta+1)^2   +\frac{\nu}{2}(2 \delta+1)   \right)   }
q_1^{\frac{1}{16} (4p + 2 \delta+1)^2 }
\int_{\R} \frac{e^{ - \frac{2 \pi w x^2}{k}}}{1- e^{ 2 \pi i r + \frac{2 \pi i  \nu}{k}- \frac{\pi i  }{2k}(2 \delta+1) -\frac{2 \pi w x}{k} }}\, dx.
\end{multline*}
Now the sum over $p$ equals
$
\frac{\eta^2 \left( \frac{1}{k}\left(h'+\frac{i}{z} \right)\right)}{\eta \left(\frac{1}{2k} \left(h'+\frac{i}{z} \right)\right)}$
thus
\begin{equation*}
\sum_2 =
\frac{\eta^2 \left( \frac{1}{k}\left(h'+\frac{i}{z} \right)\right)}{k \eta \left(\frac{1}{2k} \left(h'+\frac{i}{z} \right)\right)} \,
e^{ - \frac{ \pi i h'}{8k}  }
\sum_{\substack{ \nu \pmod k\\ \pm }}
e^{ \frac{2 \pi i h'}{k}\left(- \nu^2+\frac{\nu}{2} \right)}
\int_{\R} \frac{e^{ - \frac{2 \pi w x^2}{k}}}{1+   e^{\pm \frac{2 \pi i \nu}{k}  \mp \frac{\pi i}{2k}+2 \pi i r -\frac{2 \pi wx}{k}} }\, dx
\end{equation*}
which gives the claim using  (\ref{etatrans}),  (\ref{etaeven}), and (\ref{etaodd}).
\end{proof}
>From  Theorem \ref{transformation-1},
 we can conclude a transformation law for $\mathcal{O}(-1;q)$.
For this define
\begin{eqnarray*}
\mathcal{U}(q)&:=& \frac{ 4 \eta(z)}{\eta^2(2z)}
\sum_{ m \in \Z} \frac{q^{ \frac{1}{2}(m^2+m)}}{(1-q^m)^2}, \\
I_{ k,\nu}^{\pm}(w)&:=&\int_{\R}
\frac{e^{-\frac{2 \pi wx^2}{k} }}{\left(1+e^{\pm \frac{2 \pi i \nu}{k} \mp \frac{\pi i \widetilde{k}}{2k} - \frac{2 \pi wx}{k}}\right)^2}\, dx.
\end{eqnarray*}
\begin{corollary} \label{corollary-1}
Assume the notation above. We have.
\begin{enumerate}
\item If $k$ is odd, then
\begin{eqnarray*}
\mathcal{O}(-1;q) =
- \frac{1}{\sqrt{2}}
w^{ -\frac32} \frac{\omega_{h,k}^2}{\omega_{ 2h,k}} e^{\frac{ \pi i h'}{8k}  } \, \mathcal{U} \left( q_1^{ \frac12}\right)
+ 2 \sqrt{2} w^{\frac12 } \frac{\omega_{ h,k}^2}{k\,\omega_{2h,k}}
\sum_{ \substack{\nu \pmod k \\ \pm   }} e^{ \frac{2\pi i h'}{k}\left( - \nu^2 +\frac{\nu}{2}\right)}\, I_{k,\nu}^{\pm}(w).
\end{eqnarray*}
\item If $k$ is even, then
\begin{eqnarray*}
\mathcal{O} (-1;q) = - w^{ -\frac{3}{2}} \frac{\omega_{h,k}^2}{\omega_{ h,k/2} }\mathcal{O}(-1;q_1)
+ \frac{4 \omega_{ h,k}^2}{ k \omega_{h,k/2 }}\, w^{ \frac12}
\sum_{\nu \pmod k} (-1)^{\nu } \, e^{ -\frac{2\pi i h'\nu^2}{k}}\, I_{k,\nu}^+(w).
\end{eqnarray*}
\end{enumerate}
\end{corollary}
\begin{proof}
We only prove  (i), (ii) can be shown  similarly. We have
\begin{eqnarray*}
\frac{1}{ 2 \pi i}
\left[ \frac{\partial}{\partial r}
 \left( \frac{e^{\frac{2\pi kr^2}{w}- \frac{\pi r}{w}}  }{
 1-e^{-\frac{2\pi r}{w}  } q_1^{ \frac{m}{2}}
 }
\right)  \right]_{r=0}
 = \frac{\left(3 q_1^{ \frac{m}{2}} -1\right) }{2 i w \left(1-q_1^{\frac{m}{2} } \right)^2}.
\end{eqnarray*}
Now
\begin{eqnarray*}
\sum_{ m \in \Z}
\frac{q_1^{\frac14(m^2+m)} \left(3q_1^{\frac{m}{2} } -1\right)}{\left(1-q_1^{ \frac{m}{2}} \right)^2}
&=& - 3  \sum_{ m \in \Z}
\frac{q_1^{\frac14(m^2+m)}}{\left(1-q_1^{ \frac{m}{2}} \right)}
+ 2  \sum_{ m \in \Z}   \frac{q_1^{\frac14(m^2+m)}}{\left(1-q_1^{ \frac{m}{2}} \right)^2} \\
&=& 2 \sum_{ m \in \Z}   \frac{q_1^{\frac14(m^2+m)}}{\left(1-q_1^{ \frac{m}{2}} \right)},
\end{eqnarray*}
since in the first sum the $m$th and $-m$th term cancel.
Now (i) can be easily concluded by using that
\begin{eqnarray*}
\frac{1}{2\pi i}
\left[
\frac{\partial}{\partial r} \left(\frac{1}{1+e^{ \pm \frac{2 \pi i \nu}{k} \mp \frac{\pi i}{2k}-\frac{2\pi w x}{k}+ 2 \pi i r}          } \right) \right]_{r=0}
=  -
\frac{1}{\left(1+e^{ \pm \frac{2 \pi i \nu}{k} \mp \frac{\pi i}{2k}-\frac{2\pi w x}{k}}   \right)^2       }.
\end{eqnarray*}
\end{proof}
\begin{remark}
>From Corollary   \ref{corollary-1} we can obtain the transformation law for $\mathcal{O}(-1;z)$ under
$\left( \begin{smallmatrix} \alpha&\beta\\ \gamma& \delta  \end{smallmatrix}  \right)  \in$
$\SL_2(\Z)$  by setting $h'= \alpha$, $k= \gamma$, $h= - \delta$, and
$w=-i \left( \delta + \gamma \tau\right)$.
\end{remark}
We next realize the integral occurring in Corollary     \ref{corollary-1}  for $k=1$ as a theta integral.
For this  let
\begin{eqnarray*}
I_r^{\pm}(\tau)&:=&
\int_{\R} \frac{e^{ - \frac{2 \pi i x^2}{\tau}}}{1 \pm i e^{2 \pi i r - \frac{2 \pi i x}{\tau} }},\\
I(\tau) &:=&
\frac{1}{2 \pi  i}  \left[ \frac{\partial}{\partial r}\left(I_{r}^+(\tau) + I_{r}^-(\tau) \right) \right]_{r=0}.
\end{eqnarray*}
  \begin{lemma} \label{thetaint-1}
  We have
  \begin{eqnarray*}
  I(\tau)=   -
  \frac{(-i \tau)^2}{ \sqrt{2} \pi}
  \int_{0}^{\infty}
  \frac{\eta^2 (i u) }{\eta\left(\frac{iu}{2}\right)  (-i  ( i \tau + iu)    )^{\frac32}}\, du .
  \end{eqnarray*}
  \end{lemma}
  \begin{proof}
  Via analytic continuation it   is sufficient    to show the claim for $\tau= it$.
     Making the change of variables $x \mapsto -\frac{x}{t}$  gives
  \begin{eqnarray*}
  I_{r}^{\pm}(it)   =  t
  \int_{\R}  \frac{e^{- 2 \pi  t x^2  }}{1   \pm i  e^{2 \pi i r} \, e^{2 \pi x}}\, dx.
  \end{eqnarray*}
  We  use the theory of Mittag-Leffler to rewrite
  $$
  \frac{1}{1 + i e^{2 \pi ir} \, e^{2 \pi x} }
   +   \frac{1}{1 - i e^{2 \pi ir} \, e^{2 \pi x} }
  = - \frac{1}{2\pi}
  \sum_{m \in \Z}
  \frac{1 }{x- i \left( - m +\frac{1}{4}  -r   \right)}\,
  +  \frac{1 }{x- i \left( m -\frac{1}{4} -r   \right)}\,.
  $$
  This implies that
  \begin{eqnarray*}
  I_r^{+}(it)
  + I_{r}^{-}(it)
  = -\frac{t}{2 \pi}
  \sum_{ m \in \Z   }
  \int_{\R} \left(
  \frac{e^{-2 \pi t x^2}}{x- i  \left( -m +\frac{1}{4}  -r   \right)  }
  +  \frac{e^{-2 \pi t x^2}}{x- i  \left( m - \frac{1}{4} -r   \right)  } \right)
  \, dx.
  \end{eqnarray*}
  From this we conclude
  $$
  I(it) =  \frac{t}{4 \pi^2}
  \sum_{ \substack{m \in \Z \\ \pm  }  }
  \int_{\R}
  \frac{e^{-2 \pi t x^2}}{\left(x- i  \left( m +\frac{1}{2} \mp \frac{1}{4}    \right)  \right)^2}\, dx =
  \frac{t}{4 \pi^2}
  \sum_{ \substack{m \in \Z \\ m \text{ odd} }  }
  \int_{\R}
  \frac{e^{-2 \pi t x^2}}{\left(x- \frac{im}{4}  \right)^2}\, dx.
  $$
  We have  the integral identity
  \begin{eqnarray*}
  \int_{\R} e^{- 2 \pi u^2 t}
  \frac{u}{u-is}\, du
  = \frac{1}{2 \sqrt{2}} \int_{0}^{\infty}
  \frac{e^{ -2 \pi u s^2}}{(u+t)^{\frac32 }}\, du \qquad (s \not =0)
  \end{eqnarray*}
  which follows since both sides are   solutions of the differential equation
  $-\frac{\partial}{\partial t} + 2 \pi s^2 = \frac{1}{2 \sqrt{2} t^{\frac32 }}$ and have limit $0$ as $t \to \infty$.
  Using  integration by parts gives
    \begin{eqnarray*}
  \int_{\R}
  \frac{e^{-2 \pi tx^2}}{(x-is)^2}\, dx
  =   \sqrt{2} \pi t \int_{0}^{\infty}
  \frac{e^{-2 \pi u s^2}}{(u+t)^{\frac32}}\, du.
  \end{eqnarray*}
    Thus
  $$
  I(it)
  =- \frac{t^2}{2 \sqrt{2}\pi}
  \int_{0}^{\infty}
  \sum_{\substack{ m \in \Z \\m \text{ odd}  }}
  \frac{e^{-\frac{\pi u m^2}{8}}}{(u+t)^{\frac32}}\, du
  $$
  which easily gives the claim.
\end{proof}
Combining Theorem \ref{transformation-1}, Corollary  \ref{transformation-1},
 and Lemma \ref{thetaint-1}
  gives
        \begin{corollary} \label{transcor}
  For $z\in \H$, we have
  \begin{eqnarray*}
  \mathcal{O}\left(-1;-\frac{1}{z}\right)
  = -\frac{i}{\sqrt{2}}   (-i z)^{\frac32}  \cdot
  \mathcal{U}\left(-1;\frac{z}{2}\right)
  +   2 \frac{(-iz)^{\frac32}}{\pi}  \int_{0}^{\infty}
  \frac{\eta^2 (i u) }{\eta(iu /2)  (-i  ( i z + iu)    )^{\frac32}}\, du .
  \end{eqnarray*}
  \end{corollary}
  We next give the transformation law of the non-holomorphic part of $\mathcal{M}(-1;z)$  which can be shown as  in  proof of  Theorem  \ref{transformationnonhol}.
  \begin{lemma}
  For $z \in \H$, we have
  \begin{eqnarray*}
  J(-1;z+1)&=&J(-1;z),\\
  \frac{1}{(-iz)^{\frac32}}
  J\left(-1;-\frac{1}{z} \right) &=&
  -\frac{2i}{\pi} \int_{-\bar z}^{i \infty} \frac{\eta^2 (\tau)}{\eta(\tau/2)\, (-i (\tau+z )  )^{\frac32 }}\, d \tau
  + \frac{2i}{\pi} \int_{0}^{\infty} \frac{\eta^2 (iu)}{\eta(iu/2)\, (-i (iu+z )  )^{\frac32 }}\, du.
  \end{eqnarray*}
    \end{lemma}
    Using that  $ \frac{\eta^2(\tau)}{\eta(2\tau)}$ is a modular form on $\Gamma_0(16)$ gives  the claim.
%%%%%%%%%%%%%%%%%%%%%
%%%%%%%%%%%%%%%%%%%%%%
%%%%%%%%%%%%%%%%%%%%%%%
\section{Congruences for overpartitions} \label{SectionCong1}
Following the original strategy of Ono \cite{On1} and Ono and the
first author \cite{B2,BO1,BO2}  we can prove the congruences in
Theorem \ref{congruences}.  We limit ourselves to a sketch of the
proof.
\begin{proof}[Sketch of Proof of Theorem \ref{congruences}]
Throughout  we assume the assumptions of Theorem
\ref{congruences}. First one can show, using the results above
that the function
\begin{eqnarray}  \label{partmaass}
\sum_{n=0}^{\infty}\left(\overline{N}(r,t;n)
-\frac{\overline{p}(n)}{t}\right)q^{ n}
\end{eqnarray}
is the holomorphic part of a weak Maass form of weight
$\frac{1}{2}$  on $\Gamma_1\left(16 c^2\right)$.
  In order to use results of Serre on $p$-adic modular forms, we
  next apply twists to the associated weak Maass forms, which ``kill''
  the non-holomorphic part. This requires knowing  on which arithmetic
  progressions it is supported. We prove
\begin{equation} \label{Fourier}
\begin{split}
&\mathcal{M}\left(\frac{a}{c};z\right)=
 1
+\sum_{n=1}^{\infty}\sum_{m=-\infty}^{\infty}
\overline{N}(m,n)\zeta_c^{am}q^{n}\\
&  +  \sqrt{\pi} i  \tan \left( \frac{\pi a}{c} \right)
\sum_{k\pmod{2c}}(-1)^k
e \left(  \frac{2 ka}{c}\right)
\sum_{ m \equiv k \pmod{2c }}
\Gamma \left(
\frac{1}{2}; 4 \pi m^2y
\right)
q^{-  m^2},
\\
\end{split}
\end{equation}
where     $e(x):= e^{ 2 \pi i x}$, and where
\begin{equation}
\Gamma(a;x):=\int_{x}^{\infty}e^{-t}t^{a-1}\ dt
\end{equation}
is the incomplete Gamma-function.
Using this one can show   that for a prime $p\nmid 6t$ the function
\begin{equation} \label{twist}
\sum_{\substack{n\geq 1\\ \leg{n}{p}=-\leg{-1}{
p }}}
\left( \overline{N}(r,t;n)-\frac{\overline{p}(n)}{t}\right)q^{n}
\end{equation}
is a weight $\frac{1}{2}$ weakly holomorphic modular form on
$\Gamma_1(16 c^2 p^4 )$.
 Now the theorem can be concluded as in \cite{BO2} using a geralization of Serre's results on $p$-adic modular forms (Theorem 4.2 of \cite{BO2}).
%\begin{theorem}\label{serre}
%Suppose that $f_1(z), f_2(z),\dots, f_s(z)$ are half-integral weight
%cusp forms, where
%$$
%f_i(z)\in S_{\lambda_i+\frac{1}{2}}(\Gamma_1(N_i)) \cap \mathcal{O}_K[[q]],
%$$
%and where $\mathcal{O}_K$ is the ring of integers of a fixed
%number field $K$. If $\ell$ is prime and $j\geq 1$ is an integer,
%then the set of primes
% $p$
%for which
%$$
%f_i(z) \ | \ T\left(p^2\right)\equiv 0\pmod{\ell^j},
%$$
%for each $1\leq i\leq s$,
%has positive Frobenius density. Here $T(p^2)$ is the usual half-integer weight Hecke operator.
%\end{theorem}
\end{proof}
\section{Proof of Theorem \ref{congruences2}}  \label{SectionCong2}
Here we prove Theorem \ref{congruences2}. While we follow the
model of Ono \cite{On1} and the first author and Ono
\cite{B2,BO2}, the proof of the modularity is rather delicate. For
brevity we set $t:=\ell^m$.  Define the function
\begin{eqnarray*}
g_r(z) := t \cdot \eta^{r_1}(2 \ell z) \cdot \eta^{2 \ell} (\ell z)
\,
\sum_{n=0}^{\infty}  \overline{N}(r,t;n) q^{ n},
\end{eqnarray*}
where $0 < r_1 < 48$ is a solution of $r_1 \ell \equiv -1 \pmod{48}$. We have  \begin{equation} \label{seperate}
g_r(z) = \frac{\eta(2z) \cdot \eta^{r_1}(2 \ell z) \cdot \eta^{2 \ell}(\ell z)}{\eta^2(z)}
+ \sum_{j=1}^{t-1} \zeta_t^{-r j}
\left(\mathcal{M} \left(\frac{j}{t};z \right) + J\left(\frac{j}{t} ;z\right)\right) \eta^{r_1}(2 \ell z) \cdot \eta^{2 \ell}(\ell z).
\end{equation}
We denote the two summands by $f(z)$ and $f_r(z)$, respectively.
Define for a function $h(z)= \sum_{n} a(n) q^n$  the twist
\begin{eqnarray} \label{twist2}
\widetilde h (z) :=   -  \frac{1}{2} \leg{-1}{\ell}
 \left( h(z) - \leg{-1}{\ell}
 h(z)_{\ell}\right)_{\ell},
\end{eqnarray}
where as before
\begin{equation}\label{definetwist2}
h_{\ell}:= \frac{g}{\ell} \sum_{\nu \pmod \ell}
  \left(\tfrac{\nu}{\ell}\right)    h
  \ | \left(\smallmatrix
    1&-\frac{\nu}{\ell}\cr0&1\cr\endsmallmatrix\right).
\end{equation}
Clearly
\begin{eqnarray} \label{ftwist}
\widetilde h(z)
=
\sum_{\substack{n\\  \leg{-n}{\ell} =  -1  }} a(n) \, q^n.
\end{eqnarray}
The main step in the proof of Theorem \ref{congruences2} is the following theorem.
\begin {theorem}\label{main2}
For every  $u \geq 0$ there exist  a character $\chi$, integers $\lambda, \lambda', N,N' $, and  modular forms  $h(z) \in S_{\lambda+ \frac{1}{2}} \left( \Gamma_0\left(N\right),\chi\right)$ and
$
h_r(z) \in S_{\lambda'+\frac{1}{2}} \left(\Gamma_1\left(N' \right) \right)
$  such that
\begin{eqnarray*}
\frac{\widetilde{g_r}(z)}{\eta^{r_1}(2 \ell z) \cdot \eta^{2\ell}(\ell z)}
 \equiv h(z) + h_r(z) \pmod{\ell^u}.
\end{eqnarray*}
\end{theorem}
The proof of Theorem \ref{main2} is given later.
We first show how Theorem \ref{congruences2} follows from Theorem \ref{main2}.
\begin{proof}[Proof of Theorem \ref{congruences2}]
We easily see that
\begin{eqnarray*}
\frac{\widetilde{g_r} (z)}{\eta^{2\ell}(\ell z)\cdot \eta^{r_1}(2 \ell z)}
=  t  \sum_{\leg{-n}{\ell} = - 1}
\overline{N}(r,t;n) \, q^{n}.
\end{eqnarray*}
Now we let $0 \leq \beta \leq \ell -1$ with $ \leg{- \beta}{\ell} = -1$ be given.
Define
\begin{eqnarray*}
g_{r,\beta}(z):=
 t \sum_{n\equiv  \beta  \pmod{ \ell}}
\overline{N}\left(r,t;n\right)\,  q^{n}.
\end{eqnarray*}
Theorem \ref{main2} gives that
$$
g_{r,\beta}(z) \equiv h_{\beta}(z) +h_{r,\beta}(z) \pmod{\ell^u},
$$
where  $h_{\beta}(z)$ and $h_{r,\beta}(z)$ denote  the restrictions of the Fourier expansion of $h(z)$ resp. $h_r(z)$  to those coefficients $n$ with $n \equiv  \beta  \pmod { \ell}$.
 Using the theory of Hecke operators, we can show that  for all $n \equiv \beta \pmod {\ell}$ coprime to  $p$ we have
 \begin{eqnarray*}
  t \, \overline{N}\left(r,t; p^3 n \right) \equiv 0 \pmod{\ell^u}.
 \end{eqnarray*}
 Dividing by $t$
 directly gives the theorem since $u$ is arbitrary.
 \end{proof}
 \begin{proof}[Proof of Theorem \ref{main2}]
If $a$ is a positive integer, then  define
\begin{eqnarray*}
E_{\ell,a}(z):= \frac{\eta^{\ell^a}(z)}{\eta(\ell^az)} \in M_{\frac{\ell^a-1}{2}}  \left(\Gamma_0(\ell^a),\chi_{\ell,a} \right),
\end{eqnarray*}
where $\chi_{\ell,a}(d):= \leg{(-1)^{(\ell^a-1)/2}\ell^a}{d}$. It is well known that $E_{\ell,a}(z)$ vanishes at those cusps of $\Gamma_0(\ell^a)$ that are not equivalent to $\infty$ and that  for all $u>0$
\begin{eqnarray} \label{eisencong}
E_{\ell,a}^{\ell^{u-1}} (z) \equiv 1 \pmod {\ell^u}.
\end{eqnarray}
We now treat the summands in (\ref{seperate})  separately.
We start with $f(z)$.  Using Theorem 1.64 from \cite{On3}, it is not hard to see that
$f(z)$ is a modular form of weight $\frac{r_1+ 2 \ell -1}{2}$  with some character  on $\Gamma_0(2 \ell)$.
>From this we see    that
$\widetilde{f}(z) \in  M_{\frac{r_1+2\ell - 1}{2}}  \left( \Gamma_0(2\ell^5), \widetilde \chi\right)$ for some character $\widetilde \chi$.
For sufficiently large $u'$ the function
$$
h(z):= \frac{\widetilde{f}(z) \cdot  E_{\ell, 5}^{\ell^{u'}}(z)}{\eta^{r_1}(2 \ell z) \cdot \eta^{2 \ell}(\ell z)}
$$
is a weakly holomorphic modular form on $\Gamma_0(8 \ell^5)$ that vanishes at all cusps with the exception of $\infty$ and $\frac{1}{r \ell^5}$ with $r \in \{1,2,4\}$,
 and satisfies
$$
h(z) \equiv \frac{\widetilde{f}(z)}{\eta^{r_1}(2 \ell z) \eta^{2 \ell}(\ell z)} \pmod{\ell^u}.
$$
We show that $h(z)$ is a cusp form.
It is easy to see that  in the cusp $\infty$ the Fourier expansion of $\widetilde f(z)$ starts at least with $q^{\frac{\ell}{12}(r_1+\ell) +1}$.
Since the Fourier expansion of   $\eta^{r_1}(2 \ell z) \cdot \eta^{2 \ell}(\ell z)$ starts with
$q^{\frac{\ell}{12}(r_1+\ell) }$, we see that $h(z)$ vanishes in $\infty$.
Since $\widetilde{f}(z)$ is a modular form on $\Gamma_0\left(2 \ell^5 \right)$ and
 the Fourier expansion of $\eta^{r_1}(2 \ell z) \cdot \eta^{2 \ell}(\ell z)$ in $\frac{1}{r \ell^5}$ with $r$ even starts with $q^{\frac{\ell}{12}(r_1+\ell)}$,  $h(z)$ vanishes also in the cusps $\frac{1}{2\ell^5}$ and $\frac{1}{4 \ell^5}$.

Next consider the cusp $\frac{1}{\ell^5}$.
In (\ref{definetwist2}), we can choose a set of respresentatives with $\nu$ even.
Now for even $\nu$ it is not hard to see that
$\left(\begin{smallmatrix} 1& -\frac{\nu}{\ell}\\
0&1 \end{smallmatrix} \right)     \left(\begin{smallmatrix} 1& 0\\
\ell^5&1 \end{smallmatrix} \right)$ is $\Gamma_0(2 \ell^5)$- equivalent to
$\left(\begin{smallmatrix} 1& 0\\
\ell^5&1 \end{smallmatrix} \right)
\left(\begin{smallmatrix} 1& -\frac{\nu}{\ell}\\
0&1 \end{smallmatrix} \right)   $.
Thus
\begin{eqnarray*}
f (z)_{\ell} = \frac{g}{\ell}
\sum_{\substack{\nu\pmod \ell\\ \nu \text{ even}}}
\leg{ \nu}{\ell}
f | \left( \begin{smallmatrix} 1&0\\ \ell^5&1 \end{smallmatrix}\right)
\left( \begin{smallmatrix} 1&-\frac{\nu}{\ell}\\ 0&1 \end{smallmatrix}\right).
\end{eqnarray*}
It is not hard to see that  $f(z)$ vanishes in $\frac{1}{\ell^5}$ of order $\frac{1}{24}(-3+r_1 \ell + 4 \ell^2)$.
Since the cusp width of $\frac{1}{\ell^5}$ in $\Gamma_0(2 \ell^5)$
is $2$, the Fourier expansion of
$f | \left( \begin{smallmatrix} 1&0\\ \ell^5&1 \end{smallmatrix}\right)$ starts with $q^{r_0}$, where
$r_0:= \frac{1}{48}(-3+r_1 \ell + 4 \ell^2)$.
Thus the Fourier expansion of $\leg{-1}{\ell} f_{\ell}$ starts with
\begin{eqnarray*}
\leg{-1}{\ell} \frac{g}{\ell} \sum_{\substack{\nu \pmod \ell \\ \nu \text{ even} }}
\leg{\nu}{\ell}  e^{-\frac{2 \pi i r_0 \nu}{\ell}}
= \leg{-r_0}{\ell} =1.
\end{eqnarray*}
Since twisting doesn't decrease the order of vanishing, $\widetilde f(z)$
has in $\frac{1}{\ell^5}$ a Fourier expansion starting at least with $q^{r_0+ \frac{1}{2}}$, whereas $\eta^{r_1}(2 \ell z) \cdot \eta^{2 \ell} ( \ell z)$ has in $\frac{1}{\ell^5}$ a Fourier expansion starting with $q^{\frac{\ell(r_1+4 \ell)}{48}}$.
Thus $h(z)$ vanishes in all cusps and is therefore a cusp form.

We next turn to $f_r(z)$.
Using Theorem \ref{maassform}, it is not hard to see that $\widetilde{f_r}(z)$ is the holomorphic part of a weak Maass form on $\Gamma_1(16 t^2 \ell^4)$. Moreover by (\ref{Fourier})
it is easy to see that the corresponding weak Maass form doesn't have a  non-holomorphic part, and thus
$\widetilde{f_r}(z)$ is a weakly holomorphic modular form.
Since $E_{\ell,3m}(z)$ vanishes at each cusp $\frac{\alpha}{\gamma}$ with $t^3 \nmid \gamma$  for  sufficiently large $u'$,
the function
$$
f_r'(z):= E_{\ell,3m}^{\ell^{u'}} (z) \widetilde{f_r}(z)
$$
is a weakly holomorphic modular form on $\Gamma_1\left(16t^2\ell^4\right)$ that vanishes at all cusps $\frac{\alpha}{\gamma}$ with $t^3 \nmid \gamma$ and satisfies
$$
f_r'(z) \equiv f_r(z) \pmod{\ell^u}.
$$
 Therefore to finish the proof it remains to  show that
$
\frac{\widetilde{f_r}(z)}{\eta^{r_1}(2 \ell z) \eta^{2 \ell}(\ell z)}
$
vanishes also at those  cusps $\frac{\alpha}{\gamma}$ with $ t^3 | \gamma$. Now let
$
\left(
\begin{smallmatrix}
\alpha& \beta\\
\gamma& \delta
\end{smallmatrix}
\right)
 \in \Gamma_0( t^3).
$
In the following we need the commutation relation
for $ \nu' \equiv \delta^2 \nu \pmod \ell$
\begin{eqnarray} \label{commute}
\left(
\begin{matrix}
1 & -\frac{\nu}{\ell}\\
0&1
\end{matrix}
 \right)
 \left(
\begin{matrix}
\alpha& \beta\\
\gamma& \delta
\end{matrix}
\right)
=
  \left(
\begin{matrix}
\alpha'& \beta'\\
\gamma'& \delta'
\end{matrix}
\right)
\left(
\begin{matrix}
1 & -\frac{\nu'}{\ell}\\
0&1
\end{matrix}
 \right)
\end{eqnarray}
with
\begin{eqnarray*}
  \left(
\begin{matrix}
\alpha'& \beta'\\
\gamma'& \delta'
\end{matrix}
\right)=
  \left(
\begin{matrix}
\alpha - \frac{\gamma \nu}{\ell}& \beta - \frac{\gamma \nu \nu'}{\ell^2}+ \frac{\alpha \nu' - \delta \nu}{\ell}\\
\gamma& \delta + \frac{\gamma \nu'}{\ell}
\end{matrix}
\right) \in \Gamma_0(t^3).
\end{eqnarray*}
We distinguish the cases whether $2 | \gamma$ or not.

  If $2|\gamma$, then one can easily see that the Fourier expansion of $\eta^{r_1}(2 \ell z) \eta^{2 \ell}(\ell z)$ in $\frac{\alpha}{\gamma}$ starts with $\frac{\ell}{12}(r_1+ \ell) = :n_0$.  Clearly $\ell|n_0$.
 Thus we have to prove  that   the $q$-expansion of $\widetilde{f_r}(z)$ in $\frac{\alpha}{\gamma}$
  starts with $q^b$ with $b > n_0$.
We may assume that $48|\nu,\nu'$. Then we can write
     \begin{eqnarray*}
  (f_r)_{\ell}|
     \left( \begin{smallmatrix} \alpha&\beta\\\gamma& \delta \end{smallmatrix}  \right)
     = \sum_{ \substack{\nu \pmod {\ell}  \\ 48|\nu, \nu'   }   }
      f_r |
      \left( \begin{smallmatrix} \alpha'&\beta'\\\gamma'& \delta' \end{smallmatrix}  \right)
      \left( \begin{smallmatrix} 1&-\frac{\nu'}{\ell}\\ 0& 1 \end{smallmatrix}  \right).
     \end{eqnarray*}
     Let
      \begin{eqnarray*}
     f_{r,j}(z) := \mathcal{M} \left( \frac{j}{t};z \right) \cdot \eta^{r_1}(2 \ell z) \cdot \eta^{2 \ell} (\ell z)
     \end{eqnarray*}
     and define  $\widetilde{f_{r,j}}(z)$  for weak Maass forms as for holomorphic forms. It is enough to show that     $\widetilde{f_{r,j}}(z)$  starts with
     $q^b$ with $b >n_0$.
            As before it follows that $\widetilde{f_{r,j}}(z)$  doesn't have a non-holomorphic part.  We use twice (\ref{commute}),  Corollary   \ref{transcorollary2}, and the transformation law of the $\eta$-function.  We let $\widetilde \alpha:= (\alpha')'$,  $\widetilde \delta:= (\delta')'$,  and
     $\widetilde \gamma:= (\gamma')'$. The   holomorphic part of    $ f_{r,j}  |
      \left( \begin{smallmatrix} \widetilde \alpha&\widetilde\beta\\
      \widetilde\gamma& \widetilde\delta \end{smallmatrix}  \right) $
      is given by
            \begin{multline} \label{transfunction}
            - i  e^{  \frac{\pi i \ell}{6 \widetilde\gamma} (\widetilde\alpha  + \widetilde\delta)(r_1+\ell)  }
      e^{\frac{2 \pi i j^2 \widetilde\delta \widetilde\gamma_1}{t}}
      \frac{\omega_{\widetilde\alpha,
      \widetilde\gamma}^2}{\omega_{\widetilde\alpha, \widetilde\gamma/2} \cdot \omega_{\widetilde\alpha, \widetilde\gamma/\ell}^{2 \ell}  \cdot    \omega_{\widetilde\alpha,\widetilde\gamma/2 \ell}^{r_1}}
      \tan \left( \frac{\pi j}{t}\right) \cot \left( \frac{\pi j \widetilde\delta}{t} \right)\\
      \left(-i \left(\widetilde \gamma z + \widetilde \delta \right)  \right)^{\frac{1}{2} (1+r_1 + 2 \ell)}
      f_{r,j \widetilde\delta}(z).
      \end{multline}
      Using  \cite{Ne}
      \begin{eqnarray} \label{newman}
      \omega_{\alpha,\gamma}^{-1}
       \cdot e^{\frac{\pi i}{12 \gamma}(\alpha + \delta)} =
      \left\{
        \begin{array}{ll}
      \leg{\delta}{\gamma}\cdot  i^{\frac{1-\gamma}{2}} \cdot  e^{\frac{\pi i }{12}  (\beta \delta(1- \gamma^2) + \gamma(\alpha+ \delta) ) } & \text{if } \gamma \text{ is odd},\\
      \leg{\gamma}{\delta} e^{\frac{\pi i}{12} ( \alpha \gamma(1- \delta^2) + \delta(\beta-\gamma+3))}
       & \text{if } \delta \text{ is odd},
      \end{array}
      \right.
         \end{eqnarray}
        we can show that in (\ref{transfunction})  we can change  $\widetilde \alpha$, $\widetilde \delta$, and $\widetilde \gamma$ into $\alpha'$, $\delta'$, and $\gamma'$, respectively
         if we change $z$ into $z+ \frac{\nu'}{\ell}$.
  The Fourier expansion of $f_{r,j \delta'}$ starts with $q^{n_0}$ and  $\ell | n_0$. Moreover      \begin{eqnarray*}
      \sum_{ \substack{  \nu \pmod \ell \\ \nu \equiv 0 \pmod {48}   }  }
      \leg{ \nu}{\ell}
      e^{-\frac{2 \pi i n_0 \nu'}{\ell}}
      =  \sum_{\nu \pmod \ell} \leg{\nu}{\ell} =0.
      \end{eqnarray*}
      Thus the
     Fourier expansion of  the holomorphic part of
     $
     (f_{r,j})_{\ell}| \left(\begin{smallmatrix} \alpha'& \beta'\\ \gamma'& \delta' \end{smallmatrix} \right)
     $
     starts with at least $q^{n_0+1}$ which implies that   the Fourier expansion of
     the holomorphic part of
     $$
     \left(f_{r,j} - \leg{-1}{\ell}   (f_{r,j})_{\ell} \right)|
     \left(\begin{smallmatrix} \alpha'& \beta'\\ \gamma'& \delta' \end{smallmatrix} \right)
     $$ starts at least with $q^{n_0}$.
     Arguing in the same way, we obtain that the Fourier expansion of
     $$
     \left(
     f_{r,j} - \leg{-1}{\ell}  (f_{r,j})_{\ell}
     \right)_{\ell} |
     \left(\begin{smallmatrix}\alpha&\beta\\ \gamma& \delta  \end{smallmatrix} \right)
     $$
     starts with at least $q^{n_0+1}$ as desired.

      Next assume that $\gamma$ is odd.
     It is not hard to see that the Fourier expansion of \\
     $\eta^{r_1}(2 \ell z) \eta^{2 \ell}(\ell z)$ in $\frac{\alpha}{\gamma}$
     starts with $q^{\frac{\ell}{48} (r_1+4\ell)}=:q^{n_0}$.
      Since twisting does not decrease the order of vanishing, it is enough to show that the Fourier expansion of the holomorphic part of
      $
     \left(f_{r,j} - \leg{-1}{\ell}   (f_{r,j})_{\ell} \right)| \left(\begin{smallmatrix} \alpha& \beta \\ \gamma& \delta \end{smallmatrix} \right)
     $ starts with $q^{b}$  with $b>n_0$.
     This time we use   (\ref{commute}) once. One  can compute  that the holomorphic part of
     $ f_{r,j}|\left(\begin{smallmatrix} \alpha'& \beta'\\ \gamma'& \delta' \end{smallmatrix} \right)$ equals  \begin{multline} \label{holpart}
    \sqrt{2} i
     e^{ \frac{\pi i \ell}{6 \gamma'} \alpha' (r_1+\ell)  +
     \frac{\pi i \delta' \ell}{24 \gamma'}   (r_1+ 4 \ell)
      - \frac{\pi i  \delta'}{8 \gamma'}   + \frac{2 \pi i j^2 \delta' \gamma_1'}{t}
      }
      \cdot \tan \left(\frac{\pi j}{t} \right)
      \frac{\omega_{\alpha',\gamma'}^2 }{\omega_{2 \alpha',\gamma'} \cdot \omega_{2 \alpha', \gamma'/\ell}^{r_1} \cdot \omega_{\alpha',\gamma'/\ell}^{2 \ell}}\\
      (-i(\gamma' \tau+ \delta'))^{\frac{r_1+ 2 \ell +1}{2}}
      \cdot \eta^{2 \ell}(\ell z) \cdot \eta^{r_1} \left( \frac{\ell z}{2}  \right) \cdot
      \mathcal{U} \left( -\frac{j \delta}{t};z\right).
    \end{multline}
      Again using (\ref{newman}), one can show that one can change $\alpha'$, $\delta'$, and $\gamma'$ into $\alpha$, $\delta$, and $\gamma$, respectively   if one  changes $z$ into $z+ \frac{\nu'}{\ell}$.
      The expansion of (\ref{holpart}) starts with $q^{\frac{1}{48}  (\ell r_1 + 4 \ell^2  -3)}$.
                For \\
                $r_0:= \frac{1}{48} (\ell r_1 + 4 \ell^2 -3)$ one  clearly has $(r_0,\ell)=1$. Moreover
        \begin{eqnarray*}
        \frac{g}{\ell}  \leg{-1}{\ell} \sum_{\substack{\nu \pmod \ell\\ \nu,\nu' \equiv 0 \pmod{48}}}   \leg{\nu}{\ell} e^{-\frac{2 \pi i r_0 \nu'}{\ell}}
        = \leg{-r_0}{\ell}=1.
        \end{eqnarray*}
        Thus the expansion  of
        $
     \left(f_{r,j} - \leg{-1}{\ell}   (f_{r,j})_{\ell} \right)| \left(\begin{smallmatrix} \alpha'& \beta'\\ \gamma'& \delta' \end{smallmatrix} \right)
     $
        starts at least with $q^{\frac{1}{48}  (\ell r_1 + 4 \ell^2 +45)}$
         which implies the claim.
         \end{proof}
         %%%%%%%%%%%%%%%%%%%%%%%
         %%%%%%%%%%%%%%%%%%%%%%
         %%%%%%%%%%%%%%%%%
\section{Proof of Theorem \ref{identities}}  \label{SectionIdentities}
\begin{proof}[Proof of Theorem \ref{identities}]
>From (\ref{partmaass}) it follows  that
 \begin{eqnarray*}
 \sum_{ n = 0}^{\infty } \left(\overline{N} (s_i,\ell;n )   -\frac{ \overline{p}(n)}{ \ell}\right)\, q^n
 \end{eqnarray*}
 for $i  \in \{1,2\}$ is the holomorphic part of a weak Maass form on $\Gamma_1(16 \ell^2)$.
 Moreover   the non-holomorphic  part    is supported on negative squares. The same is true for the function
 \begin{eqnarray*}
 \sum_{ n=0}^{ \infty} \left( \overline{N}(s_1,\ell;n) - \overline{N} (s_2,\ell;n)\right)q^n.
 \end{eqnarray*}
 The restriction of the associated weak Maass form to those coefficients congruent to $d$ modulo $\ell$ gives a weak Maass form on $\Gamma_1(16 \ell^4)$.
 Since $\leg{d}{\ell} = - \leg{-1}{\ell}$ it does not have a non-holomorphic part which proves Theorem \ref{identities}.
 \end{proof}
 Next we state some identities which may be deduced thanks to our   theorem.
\noindent
Define for a positive integer $N$, $g,h$ real numbers that are not simultaneously congruent to $0 \pmod N$,
the generalized Dedekind eta-function
\begin{eqnarray*}
E_{g,h}(z) :=
q^{ \frac{1}{2} B \left(\frac{g}{N}  \right)}
\cdot
\prod_{m=1}^{\infty}
\left(
1-  \zeta_{N}^h \cdot  q^{m-1+ \frac{g}{N}}
\right)
\left(
1-  \zeta_{N}^{-h} \cdot  q^{m- \frac{g}{N}}
\right),
\end{eqnarray*}
where $B(x):=x^2-x+ \frac{1}{6}$. We have the following identities.
\begin{equation} \label{identity1}
\sum_{n=0}^{\infty}  \left(\overline{N}(1,5,5n+2) - \overline{N}(2,5,5n+2) \right)\, q^{5n+2} = 2 \frac{\eta(50z)}{E_{1,0}(25z)},
\end{equation}
\begin{equation}
\sum_{n=0}^{\infty}  \left( \overline{N}(1,5,5n+3) - \overline{N}(2,5,5n+3) \right)\, q^{5n+3} = -2 \frac{\eta(50z)}{E_{2,0}(25z)},
\end{equation}
\begin{equation}
\sum_{n=0}^{\infty}  \left(  \overline{N}(0,5,5n+3) - \overline{N}(2,5,5n+3) \right)\, q^{5n+3} = 2 \frac{\eta(50z)}{E_{2,0}(25z)},
\end{equation}
\begin{equation}
\sum_{n=0}^{\infty}  \left( \overline{N}(0,5,5n+2) - \overline{N}(2,5,5n+2) \right)\, q^{5n+2} =0,
\end{equation}
\begin{equation}
\sum_{n=0}^{\infty}  \left( \overline{N}(0,3,3n+1) -  \overline{N}(1,3,3n+1) \right)\, q^{3n+1} =
2\frac{\eta(9z) \cdot \eta(18 z)}{\eta(3z)}.
\end{equation}

\end{document}